\documentclass[12pt,a4paper,reqno]{amsart} 
\usepackage{ae,aecompl}
\usepackage{amssymb}
\usepackage[english]{babel}
\usepackage{pictex,dcpic}

\numberwithin{equation}{section}
\allowdisplaybreaks[4]

\makeatletter
\renewcommand*\env@matrix[1][c]{\hskip -\arraycolsep
  \let\@ifnextchar\new@ifnextchar
  \array{*\c@MaxMatrixCols #1}}
\makeatother

\title[ASD Connections on $\CP^2_q$: Instantons]{\vspace*{-1cm}Anti-selfdual Connections on the\\[10pt] Quantum Projective Plane: Instantons \\[20pt]}

\date{May 2013; v2 August 2014}

\dedicatory{Dedicated to the memory of Tetsuya Masuda}

\author[F.~D'Andrea]{Francesco D'Andrea} 
\address[F.~D'Andrea]{Dipartimento di Matematica e Applicazioni, Universit\`a di Napoli Federico II, Piazzale Tecchio 80, I-80125 Napoli, Italy, and INFN, Sezione di Napoli, Napoli, Italy}
\email{francesco.dandrea@unina.it}

\author[G.~Landi]{Giovanni Landi \\[20pt]}
\address[G.~Landi]{
Matematica, Universit\`{a} di Trieste, Via A.~Valerio~12/1, I-34127 Trieste, Italy, and INFN, Sezione di Trieste, Trieste, Italy}
\email{landi@units.it}

\thanks{\rule{0pt}{10pt}\textit{Acknowledgments.}
Both authors were partially supported by the Italian Project ``Prin 2010-11 -- Operator Algebras, Noncommutative Geometry 
and Applications''. FD was partially supported by UniNA and Compagnia di San Paolo under the grant ``STAR Program 2013''.}

\newtheorem{prop}{Proposition}[section]
\newtheorem{lemma}[prop]{Lemma}
\newtheorem{cor}[prop]{Corollary}
\newtheorem{thm}[prop]{Theorem}

\theoremstyle{definition}
\newtheorem{rem}[prop]{Remark}

\newcommand{\A}{\mathcal{A}}

\newcommand{\E}{\mathcal{E}}
\newcommand{\F}{F}
\newcommand{\U}{\mathcal{U}}
\newcommand{\HH}{\mathcal{H}}

\DeclareMathOperator{\SU}{SU}

\newcommand{\Z}{\mathbb{Z}}
\newcommand{\R}{\mathbb{R}}
\newcommand{\C}{\mathbb{C}}

\newcommand{\CP}{\mathbb{C}\mathrm{P}}
\newcommand{\Aq}{\mathcal{A}(\mathbb{C}\mathrm{P}^2_q)}
\newcommand{\Oq}{\mathcal{A}(\mathrm{SU}_q(3))}
\newcommand{\Sq}{\mathcal{A}(\mathrm{S}^5_q)}
\newcommand{\Uq}{\mathcal{U}_q(\mathfrak{su}(3))}
\newcommand{\Kq}{\mathcal{U}_q(\mathfrak{u}(2))}

\newcommand{\id}{\textup{id}}
\newcommand{\az}{\triangleright}
\newcommand{\za}{\triangleleft}
\newcommand{\mL}[1]{\mathcal{L}_{#1}}

\newcommand{\maa}[1]{\bigg(\!\begin{array}{cc}#1\end{array}\!\bigg)}
\newcommand{\tr}{\mathrm{Tr}}

\newcommand{\wprod}{\wedge_q\mkern-1mu}
\newcommand{\de}{\partial}
\newcommand{\deb}{\bar{\partial}}
\newcommand{\dd}{\mathrm{d}}
\newcommand{\ii}{\mathrm{i}}
\newcommand{\Hs}{\star_H\hspace{0.5pt}}
\newcommand{\vol}{\tau}
\newcommand{\rPN}{r_N}

\newcommand{\kahler}{\omega}
\newcommand{\qkahler}{\omega_q}

\begin{document}

\begin{abstract}
We study one-instantons over $\overline{\CP}{}^2_q$, that is anti-selfdual
connections with instanton number $1$ on the quantum projective plane $\CP^2_q$ with 
orientation which is reversed with respect to the usual one.
The orientation is fixed by a suitable choice of a basis element  
for the rank $1$ free bimodule of top forms.
The noncommutative family of solutions is foliated, each non-singular
leaf being isomorphic to $\overline{\CP}{}^2_q$ itself.
\end{abstract}

\maketitle

\vspace{-5mm}

{
\renewcommand{\contentsname}{{\normalsize Contents}}
\small
\baselineskip=10pt
\tableofcontents
}

\vspace*{-1cm}


\pagebreak

\section{Introduction}\label{sec:1}
For Yang-Mills theory, two most used manifolds, the four-dimensional sphere and the complex projective plane, 
show a marked different behaviour. This is because unlike the four-sphere, the projective plane has no orientation-reversing isometry, 
and, for this theory, one has two distinct oriented manifolds, the plane with standard orientation --- denoted $\CP^2$ --- and the same manifold with reversed orientation --- denoted $\overline{\CP}{}^2$. 
A clear manifestation of the consequences 
is the fact that the moduli space of $\SU(2)$ one-instantons, that is to say anti-selfdual (ASD) connections 
on a $\SU(2)$-vector bundle $E$ over 
$\CP^2$ with second Chern number $c_2(E) =1$, is empty, while this is not the case for $\overline{\CP}{}^2$. Indeed for the latter oriented manifold the moduli space \cite{Buc86,Don84} (cf.~also \cite{Gro90} and \cite{Hab92}) has very interesting geometrical structures: it turns out to be a open cone over $\overline{\CP}{}^2$ (of course anti-selfdual connections on $\overline{\CP}{}^2$ are the same as selfdual connections on $\CP^2$). 

The work started in \cite{DL09b} was devoted to monopole connections on the quantum projective plane $\CP^2_q$, that is ASD connection on line bundles 
over $\CP^2_q$. We continue here with one-instantons, meaning with this ASD connections on a rank $2$ 
complex vector bundle with instanton number $1$.  
The base space of the bundle is taken to be $\overline{\CP}{}^2_q$, that is the quantum projective plane $\CP^2_q$ with orientation which is reversed with respect to the usual one. The orientation is fixed by a suitable choice of a volume form.   
The noncommutative space of ASD solutions is foliated, each non-singular leaf being isomorphic to $\overline{\CP}{}^2_q$  itself. 

As for the classical case, the module $\Omega^{2}$ of two-forms on the quantum projective plane can be decomposed as a 
direct sum of four submodules, 
$$
\Omega^{2}=\Omega^{2,0}\oplus\Omega^{1,1}_v\oplus\Omega^{1,1}_s \oplus\Omega^{0,2} \:, 
$$
where $\Omega^{1,1}_s$ is a rank $1$ free bimodule with basis the $q$-analogue of the K{\"a}hler form.
Classically, on $\CP^2$ with standard orientation a two-form is ASD if and only if
it belongs to $\Omega^{1,1}_v$ (sections of a rank $3$ vector bundle on $\CP^2$). 
On the space $\overline{\CP}{}^2$,
obtained from $\CP^2$ by reversing the orientation, a two-form is ASD if and only if
its component in $\Omega^{1,1}_v$ is zero, that is its $(1,1)$ component
is proportional to the K{\"a}hler form (see e.g.~\cite{DK90}). When $q\neq 1$,  
something similar happens.
Using the canonical Hermitian structure on the modules of forms, we introduce a Hodge star operator which depends on the choice of a basis for the rank $1$ free bimodule of top forms $\Omega^{2,2}$; the basis element, that we interpret as ``volume form'', is unique modulo a rescaling of $\lambda\in\R\smallsetminus\{0\}$, and
the condition that $\Hs^2=(-1)^k\id$ on $k$-forms determines the volume form up to a sign $\lambda=\pm 1$. The sign, that we interpret as a choice of orientation, determines whether an ASD two-form is in $\Omega^{1,1}_v$ or 
$\Omega^{2,0}\oplus\Omega^{1,1}_s\oplus\Omega^{0,2}$. With our conventions, the choice $\lambda=1$ gives the standard orientation, or $\CP^2_q$, while $\lambda=-1$ is for the reverse orientation, or $\overline{\CP}{}^2_q$. As it happens classically, 
an ASD connection (on a finitely projective module --- a bundle) over $\overline{\CP}{}^2_q$ is one whose curvature has $(1,1)$ component proportional to the K{\"a}hler form.

The one-instanton connections 
on $\overline{\CP}{}^2_q$ that we present here form a noncommutative family 
that is the analogue of an open cone over
the quantum projective plane itself.
For each value of a real parameter $t\in [0,1)$, the coordinate on a fixed generator of the cone ($t=0$ yielding
the vertex), one has an ASD connection $\smash[t]{\widetilde{\nabla}_t}\rule{0pt}{11pt}$ with connection one-form:
$$
\omega_t=\begin{pmatrix}
\;A^*-A & \Phi\; \\[2pt]
\;-\,\Phi^* & B-B^*
\end{pmatrix}
$$
where $A,B\in\Omega^{0,1}$ and $\Phi$ is an $\Omega^1$-valued endomorphism of the bundle,  
given explicitly by 
\begin{align*}
A^* &=q\sqrt{1-t^2 \;p_{11} }\;\de \frac{1}{\sqrt{1-t^2 \;p_{11} }} \;, \qquad
B =q^{-1}\sqrt{1-t^2q^4 \;p_{11} }\; \deb \frac{1}{\sqrt{1-t^2q^4 \;p_{11} }}
\;, \\
\Phi &= t\,
\frac{1}{\sqrt{1-t^2 \;p_{11} }} \, \left(\sum\nolimits_j\bigl\{z_j(\de p_{j2})z_3-qz_j(\de p_{j3})z_2\bigr\} \right) 
\,\frac{1}{\sqrt{1-t^2q^4 \;p_{11} }}
\;.  
\end{align*}
Here the $\{p_{jk}\}$'s are the generators of the algebra $\Aq$, the $\{z_{j}\}$'s generate the algebra $\Sq$ of a covering quantum sphere and $\de$, $\deb$ are a holomorphic and corresponding antiholomorphic exterior derivatives over $\CP^2_q$.
For each $t$ a family of connections is obtained from $\smash[t]{\widetilde{\nabla}_t}$ by using the coaction of the symmetry quantum group $\SU_q(3)$.
For each fixed $t>0$, the family is parametrized by the quantum projective plane itself, or rather its algebra $\Aq$.
When $t=0$ the connection `reduces' to the direct sum 
$$
\widetilde{\nabla}_0:=\nabla_1\oplus\nabla_{-1}
$$
of the monopole and antimonopole connections of \cite{DL09b} on the direct sum of line bundles $L_1\oplus L_{-1}$ 
(in analogy with the case $q=1$, the bundle $L_{-1}$ is the ``module of sections" of the tautological line bundle, and $L_1$ is its dual). Since the connection $\widetilde{\nabla}_0$ is $\SU_q(3)$-coinvariant, for $t=0$ we have a single connection 
(the vertex of the cone). For $t\neq 0$ the connection $\widetilde{\nabla}_t$, which is still defined on the bundle $L_1\oplus L_{-1}$, 
is irreducible, meaning that there is no non-trivial submodule of $L_1\oplus L_{-1}$ that is preserved by the connection.

The reason to call ``one-instanton'' the ASD connection $\widetilde{\nabla}_0$, as well as the more general ones,  is due to the fact that the bundle $L_1\oplus L_{-1}$ on which they are all defined, has the correct ``topological numbers", i.e.~rank $2$, charge $0$, and instanton number $1$, respectively.

\medskip

\noindent{\bf Notations.}\\
Throughout this paper, by a $*$-algebra we always mean a unital associative involutive complex algebra, whose representations will be implicitly assumed to be unital $*$-representations and the representation symbols will be omitted.
The real deformation parameter will be taken to be $0<q<1$. We denote by
$$
[z]_q:=\frac{q^z-q^{-z}}{q-q^{-1}}
$$
the $q$-analogue of a number $z\in\C$, we define recursively the $q$-factorial by $[0]!:=1$ and
$[n]_q!:=[n]_q[n-1]_q!$ for $n\geq 1$, and finally the $q$-trinomial coefficient by:
$$
[j,k,l]_q!=q^{-(jk+kl+lj)}\frac{[j+k+l]_q!}{[j]_q![k]_q![l]_q!} \;.
$$
We use Sweedler notation for the coproduct, $\Delta(a) = a_{(1)}\otimes a_{(2)}$ with a sumation understood,
and write the opposite coproduct as $\Delta^{\mathrm{cop}}(a) = a_{(2)}\otimes a_{(1)}$.


\section{The base space and the bundles}\label{sec:2}
In this section, we recall the definion and some of the properties of
the quantum complex projective plane $\CP^2_q$ along the lines of the papers \cite{DDL08b,DL09b,DL08}, using in particular
the notations of \cite{DL09b}.  This is defined as a $q$-deformation of the complex projective plane $\CP^2$
seen as the real manifold $\SU(3)/\mathrm{U}(2)$. We start then from deformations of Lie groups and Lie algebras.


\subsection{The quantum group $\SU_q(3)$ and its homogeneous spaces}
Let $\Uq$ be the compact real form of the Hopf algebra denoted $\breve{U}_q(\mathfrak{sl}(3))$ in Sec.~6.1.2 of~\cite{KS97}.
As a $*$-algebra it is generated by elements $\{K_i,K_i^{-1},E_i,F_i\}_{i=1,2}$,
with $*$-structure $K_i=K_i^*$ and $F_i=E_i^*$, $i=1,2$, and relations
\begin{gather*}
[K_i,K_j]=0\;,\qquad
[E_i,F_i]=\frac{K_i^2-K_i^{-2}}{q-q^{-1}}\;,  \qquad [E_i,F_j]=0 \quad\mathrm{if}\;i\neq j \;, \\
K_iE_iK_i^{-1}=qE_i\; \qquad K_iE_jK_i^{-1}=q^{-1/2}E_j \quad\mathrm{if}\;i\neq j \;, \\
[E_i,[E_j,E_i]_q]_q=0 \;.
\end{gather*}
Here the symbol $[a,b]_q$ denotes the $q$-commutator of operators $a,b$, i.e. $[a,b]_q:=ab-q^{-1}ba$.
It becomes a Hopf $*$-algebra with the following coproduct, counit and antipode:
\begin{gather*}
\Delta(K_i)=K_i\otimes K_i\;,\quad
\Delta(E_i)=E_i\otimes K_i+K_i^{-1}\otimes E_i\;, \quad 
\Delta(F_i)=F_i\otimes K_i+K_i^{-1}\otimes F_i\;, \\
\epsilon(K_i)=1\;,\qquad
\epsilon(E_i)=\epsilon(F_i)=0\;, \\
S(K_i)=K_i^{-1}\;,\qquad
S(E_i)=-qE_i\;,\qquad
S(F_i)=-q^{-1}F_i\;,
\end{gather*}
for $i=1,2$.
For obvious reasons we denote by $\U_q(\mathfrak{su}(2))$ the Hopf $*$-subalgebra
of $\Uq$ generated by the elements $\{K_1,K_1^{-1},E_1,F_1\}$, while 
$\Kq$ denotes the Hopf $*$-subalgebra generated by $\U_q(\mathfrak{su}(2))$ together with $K_1K_2^2$
and $(K_1K_2^2)^{-1}$.

\smallskip

On
the collection $\Uq'$ of linear maps $\Uq\to\C$  one defines operations dual to those of $\Uq$
as follows. For $f,g:\Uq\to\C$, the product is  
$$
(f\cdot g)(x):= (f \otimes g)(\Delta x) = f(x_{(1)})g(x_{(2)}) \;,
$$
for all $x\in\Uq$, and $\Delta x =  x_{(1)} \otimes x_{(2)}$ in Sweedler notation. 
The unit is the map $1(x):=\epsilon(x)$. The coproduct, counit, antipode and
$*$-involution are given by
\begin{align*}
\Delta(f)(x,y) &:=f(xy) \;,& \hspace{-15mm}
\epsilon(f)&:=f(1) \;,\\
S(f)(x) &:=f(S(x)) \;,& \hspace{-15mm}
f^*(x)&:=\overline{f(S(x)^*)} \;,
\end{align*}
for all $x,y\in\Uq$, and with $\bar{c}$ the complex conjugate of $c\in\C$.
As usual, they satisfy all the axioms of a Hopf $*$-algebra, except that $\mathrm{Im}(\Delta)\subset\mathrm{Hom}_{\C}(\Uq\otimes\Uq,\C)$
is bigger than $\Uq'\otimes\Uq'$.

A proper Hopf $*$-algebra $\Oq\subset\Uq'$ is the one generated by the matrix elements $u^i_j$ and $(u^i_j)^*$, 
with $i,j=1,2,3$, of the fundamental representations of $\Uq$.
As an abstract Hopf $*$-algebra, it is defined by the commutation relations (cf.~\cite{KS97}, Sec.~9.4):
\begin{align*}
u^i_ku^j_k &=qu^j_ku^i_k \;,&
u^k_iu^k_j &=qu^k_ju^k_i \;,&&
\forall\;i<j\;, \\
[u^i_l,u^j_k]&=0 \;,&
[u^i_k,u^j_l]&=(q-q^{-1})u^i_lu^j_k \;,&&
\forall\;i<j,\;k<l\;,
\end{align*}
and by a cubic relation:
$$
\sum\nolimits_{\pi\in S_3}(-q)^{l(\pi)}u^1_{\pi(1)}u^2_{\pi(2)}u^3_{\pi(3)}=1 \;,
$$
where the sum is over all permutations $\pi$ of the three elements
$\{1,2,3\}$ and $l(\pi)$ is the number of inversions in $\pi$.
The $*$-structure is given by
\begin{equation}\label{eq:star}
(u^i_j)^*=(-q)^{j-i}(u^{k_1}_{l_1}u^{k_2}_{l_2}-qu^{k_1}_{l_2}u^{k_2}_{l_1}) \;,
\end{equation}
with $\{k_1,k_2\}=\{1,2,3\}\smallsetminus\{i\}$ and
$\{l_1,l_2\}=\{1,2,3\}\smallsetminus\{j\}$, as ordered sets.
As expected for a corepresentation, coproduct, counit and antipode are of `matrix' type:
$$
\Delta(u^i_j)=\sum\nolimits_ku^i_k\otimes u^k_j\;,\qquad
\epsilon(u^i_j)=\delta^i_j\;,\qquad
S(u^i_j)=(u^j_i)^*\;.
$$

The algebra $\Oq$ is a bimodule $*$-algebra for the left and right canonical
actions of $\Uq$, denoted $\az$ and $\za$ respectively and given by:
$$
(x\az f)(y):=f(yx) \qquad\mathrm{and}\qquad (f\za x)(y):=f(xy) \;,
$$
for all $f\in\Oq$ and all $x,y\in\Uq$. Explicitly, on generators:
\begin{align*}
K_i\az u^j_k &=q^{\frac{1}{2}(\delta_{i+1,k}-\delta_{i,k})}u^j_k \;,&
E_i\az u^j_k &=\delta_{i,k} u^j_{i+1}\;, &
F_i\az u^j_k &=\delta_{i+1,k} u^j_i\;, \\
u^j_k\za K_i &=q^{\frac{1}{2}(\delta_{i+1,j}-\delta_{i,j})}u^j_k \;,&
u^j_k\za E_i &=\delta_{i+1,j} u^i_k \;, &
u^j_k\za F_i &=\delta_{i,j} u^{i+1}_k \;.
\end{align*}

\smallskip

The algebras of ``functions" on the quantum sphere $\mathrm{S}^5_q$
and on the quantum projective plane $\CP^2_q$ are defined, respectively,
as the fixed point subalgebras of $\Oq$ for the right canonical action of $\U_q(\mathfrak{su}(2))$
and $\Kq$,
$$
\Sq:=\Oq^{\U_q(\mathfrak{su}(2))} \;,\qquad
\Aq:=\Oq^{\Kq} \;,
$$
and are left $\Uq$-module $*$-algebras for the restrictions of the left canonical
action. 

Generators of $\Sq$ are the elements $z_i:=u_i^3$ and generators of
$\Aq$ are the elements $p_{ij}:=z_i^*z_j$. The former algebra is generated,
as an abstract $*$-algebra by elements $\{z_i,z_i^*\}_{i=1,2,3}$ with relations~\cite{VS91}:
\begin{gather}
z_iz_j=qz_jz_i\quad\forall\;i<j \;,\qquad\quad
z_i^*z_j=qz_jz_i^*\quad\forall\;i\neq j \;, \notag\\
\rule{0pt}{16pt}
[z_1^*,z_1]=0 \;,\qquad
[z_2^*,z_2]=(1-q^2)z_1z_1^* \;,\qquad
[z_3^*,z_3]=(1-q^2)(z_1z_1^*+z_2z_2^*) \;, \notag\\
\rule{0pt}{16pt}
z_1z_1^*+z_2z_2^*+z_3z_3^*=1 \;.\label{eq:sphere}
\end{gather}
The elements $p_{ij}$ generating the latter algebra $\Aq$ can be arranged
as matrix entries in a projection $p$ that we name the ``defining'' projection.
It obeys \cite[pg.~848]{DL09b}
$$
q^4p_{11}+q^2p_{22}+p_{33}=1 \;,
$$
a $q$-trace condition to be used later on.
For $q=1$, we get a commutative algebra
generated by the matrix entries of a size $3$ and rank $1$ complex projection;
the underlying space is diffeomorphic (as a real manifold) to the projective plane
$\CP^2$ upon identifying each line through the origin in $\C^3$ with the range of a projection.
The generators of $\Sq$ play the role of ``homogeneous coordinates" of $\CP^2_q$
in this non-commutative setting.

For future use, we record the `orthogonality' relations
for rows and columns of $(u^i_j)$.
\begin{lemma}\label{lemma:ortho}
For all $a,b=1,2,3$ we have:
\begin{align*}
\sum\nolimits_iu^a_i(u^b_i)^* &=\delta_{a,b} \;, &
\sum\nolimits_iq^{2(a-i)}(u^a_i)^*u^b_i &=\delta_{a,b} \;,\\
\sum\nolimits_iq^{2(i-b)}u^i_a(u^i_b)^* &=\delta_{a,b} \;, &
\sum\nolimits_i(u^i_a)^*u^i_b &=\delta_{a,b} \;.
\end{align*}
\end{lemma}
\begin{proof}
These relations follow from \cite[Prop.~9.2.8]{KS97}, using $(u^j_k)^*=(-q)^{k-j}A^j_k$ and the quantum determinant condition $\mathcal{D}_q=1$.
\end{proof}

\begin{rem}\label{rem:otherS}
The two $*$-subalgebras of $\Oq$ generated by the first row or the first column respectively, of the matrix
$(u^i_j)$ are both isomorphic to $\Sq$ as can be easily seen.
In the first case the isomorphism is given on generators by the map
$u_1^1\mapsto (z_3)^*$, $u^1_2\mapsto (z_2)^*$ and $u^1_3\mapsto (z_1)^*$;
in the second one by the map $u_1^1\mapsto (z_3)^*$, $u_1^2\mapsto (z_2)^*$ and $u_1^3\mapsto (z_1)^*$.
\end{rem}


\subsection{Equivariant vector bundles}\label{sec:evb}

It is computationally useful to transform the right action $\za$ to a left one 
$\mL{x}a:=a\za S^{-1}(x)$, still commuting with the action $\az$.
The presence of the antipode yields the generalized Leibniz rule:
\begin{equation}\label{glr}
\mL{x}(ab)=(\mL{x_{(2)}}a)(\mL{x_{(1)}}b) \;,
\end{equation}
for all $x\in\Uq$ and $a,b\in\Oq$.
Let $\sigma:\Kq\to\mathrm{End}(\C^n)$ be an $n$-dimensional $*$-representation.
The analogue of (sections of) the equivariant vector bundle associated to $\sigma$
is the collection $\E(\sigma)$ of elements of $\Oq\otimes\C^n$ that are $\Kq$-invariant
for the Hopf tensor product of the actions $\mL{}$ in \eqref{glr} and $\sigma$: 
\begin{align}\label{eq:Esigma}
\E(\sigma) &:= \Oq\!\boxtimes_{\sigma} \C^n \nonumber \\
&:= \big\{\psi \in \Oq \otimes \C^n ~\big|~ \big(\mL{h_{(1)}} \otimes \sigma(h_{(2)})\big)(\psi) = \epsilon(h) \psi \,;\;\; \forall\;h\in\Kq\big\} \;.
\end{align}
As $\E (\sigma)$ is stable under (left and right)
multiplication by $\U_q(\mathfrak{su}(2))$-invariant elements of $\Oq$, we have that $\E(\sigma)$ is an
$\Aq$-bimodule. Since, as mentioned, the actions $\mL{}$ and $\az$ commute, it is also a left $\Aq\rtimes\Uq$-module.

The vector space $\E(\sigma)$ is a left $\Oq$-comodule as well.
It is a well known and general fact that the coproduct of a Hopf algebra
defines two mutually commuting coactions of the algebra upon itself,
called the left and right regular coactions.
The left regular coaction commutes with the left canonical action of any dual Hopf algebra,
and similarly for the pair of right action/coaction. As for the action, it is useful to turn the right regular
coaction of $\Oq$ on itself into a left coaction $\Delta_L$ using the antipode. We define:
\begin{equation}\label{eq:DeltaL}
\Delta_L=(\id\otimes S)\Delta S^{-1} \;.
\end{equation}
Using the properties of a Hopf algebra, one easily checks that this is indeed a left coaction,
$$
(\id\otimes\Delta_L)\Delta_L=(\Delta\otimes\id)\Delta_L \;,
$$
although it is not an algebra morphism due to the presence of the antipode;
so $\Oq$ with $\Delta_L$ is a comodule but not a comodule-algebra.
For $x,y,z\in\Uq$ we have
$$
\Delta_L(f\za x)(y,z)=f(xzS^{-1}(y))=\Delta_L(f)(y,xz)= 
\big( f_{(\bar 1)}\otimes (f_{(\bar 2)}\za x)\big)(y,z) \;,
$$
with $\Delta_L(f) = f_{(\bar 1)}\otimes f_{(\bar 2)}$ in Sweedler notations. This explicitly proves
commutativity of the coaction with the action, that is,
$$
\Delta_L\circ\mL{x}=(\id\otimes\mL{x})\Delta_L
\;,\qquad \forall\;x\in\Uq.
$$
The coaction is extended trivially to $\Oq\otimes\C^n$, and from the previous commutativity we deduce
that the subspace $\E(\sigma)$ is a left $\Oq$-subcomodule.

\subsection{Line bundles and their characteristic classes}\label{sec:line}
Let us give few additional details about the modules that will be used in the next sections.
First of all, we can observe that since $\E (\sigma_1\oplus\sigma_2)\simeq\E(\sigma_1)\oplus\E(\sigma_2)$, it is enough
to focus on irreducible representations of $\Kq$. The irreducible representations $\sigma_{\ell,N}$ that appear in the decomposition
of $\Oq$ are classified by two half-integers $\ell$ and $N$, the \emph{spin} and the \emph{charge}, with the constraints $\ell\geq 0$ and $\ell+N\in\Z$
(cf.~Sec.~2.1 of \cite{DL09b}). For $q=1$ the vector bundle associated to $\sigma_{\ell,N}$ has rank equal to the dimension of the representation, that is $2\ell+1$. 

In particular, the analogue of line bundles are the bimodules
$$
L_N:=\E(\sigma_{0,N})\;, \quad N\in\Z\;
$$
(these were denoted $\Sigma_{0,N}$ in \cite{DL09b}). Projectivity, as one sided modules, can be
explicitly proved as follows. Let $\Psi_N$ be the column vector with components defined by:
\begin{align*}
\psi_{j,k,l}^N &:=\sqrt{[j,k,l]!}\,(z_1^jz_2^kz_3^l)^* \;,
   && \textup{if}\;N>0\; \quad\textup{and with}\quad \; j+k+l=N\,,\\[2pt]
\psi_{j,k,l}^N &:=q^{-N+j-l}\sqrt{[j,k,l]!}\,z_1^jz_2^kz_3^l \;,
   && \textup{if}\;N<0\; \quad\textup{and with}\quad \; j+k+l=-N\,.
\end{align*}
With this, construct a $\rPN \times \rPN$ projection \cite{DL09b}:
\begin{equation}\label{mon-pro}
P_N:=\Psi_N\Psi_N^\dag \;,
\end{equation}
with rows/columns in number $\rPN:=\frac{1}{2}(|N|+1)(|N|+2)$. As shown in~\cite{DL09b}, the map
\begin{subequations}
\begin{equation}\label{eq:isoPsiNa}
L_N\to [\Aq]^{\rPN}P_{-N}\;,\qquad a\mapsto a\Psi^\dag_{-N} \; ,
\end{equation}
is an isomorphism of left $\Aq$-modules, while the map
\begin{equation}\label{eq:isoPsiN}
L_N\to P_N [\Aq]^{\rPN}\;,\qquad a\mapsto \Psi_Na \; ,
\end{equation}
\end{subequations}
is an isomorphism of right $\Aq$-modules. In particular what we have named the defining projection is just the projection $P_{-1}$.

Finitely generated projective $\Aq$-modules are parametrized by three integers, coming from maps
$$
\mathrm{ch}^0_{(\pi_k,\HH_k,F_k)}:K_0(\Aq)\to\Z \;, \quad k=0,1,2 \;,
$$
associated to three $1$-summable Fredholm modules $(\pi_i,\HH_i,F_i)$.
The pairing with a projection gives `rank', `charge' and
`instanton number' of the associated vector bundle.
For line bundles, it was computed in \cite[Prop.~4.1]{DL09b} that for any $N\in\Z$:
$$
\mathrm{ch}^0_{(\pi_0,\HH_0,F_0)}([P_N])=1 \;,\quad
\mathrm{ch}^0_{(\pi_1,\HH_1,F_1)}([P_N])=N \;,\quad
\mathrm{ch}^0_{(\pi_2,\HH_2,F_2)}([P_N])=\tfrac{1}{2}N(N+1) \;,
$$
Since these maps are additive,
the direct sum $L_1\oplus L_{-1}$ ---
i.e.~the module corresponding to the projection $P_1\oplus P_{-1}$
--- is the analogue of a vector bundle with rank $2$, charge $0$ and instanton number $1$:
this is the module on which we will construct one-instantons.

\section{The differential structure}

The ``differential'' or ``smooth'' structure of $\CP^2_q$ is described by a differential graded $*$-algebra $(\Omega^\bullet,\dd)$,
with $\Omega^0=\Aq$.
Thus, $\Omega^\bullet=\bigoplus_{k=0}^4\Omega^k$ is a graded associative \mbox{$*$-algebra},
and $\dd:\Omega^\bullet\to\Omega^{\bullet+1}$ is a graded derivation with square zero: $\dd^2=0$. 
In addition, $\dd(\omega^*)=(\dd\omega)^*$
--- i.e.~we have a \emph{real} differential calculus, or a \emph{$*$-calculus} --- and the algebra of forms is generated
by forms of degree $0$ and $1$ (cf.~\cite[Lemma~5.3]{DL09b}).

\begin{rem}\label{rem:2.1}
Here we choose to have $\dd(a^\ast):=(\dd a)^\ast$, instead of
$\dd(a^\ast):=-(\dd a)^\ast$ like in \cite{DL09b} (of course,
one can pass from one notation to the other multiplying the differential
by $\ii=\sqrt{-1}$). With this notation, $\dd$ is a Hermitian connection on free modules.
\end{rem}

In parallel with the classical $\CP^2$ being a complex manifold, differential forms on $\CP^2_q$
form a double complex denoted by $\Omega^{\bullet,\bullet}(\CP^2_q)$ or simply
$\Omega^{\bullet,\bullet}$ \cite{DL09b}. The space of $k$-forms decompose as $\Omega^k=\bigoplus_{i+j=k}\Omega^{i,j}$,
so $\Omega^{\bullet,\bullet}=\bigoplus_k\Omega^k=\bigoplus_{i,j}\Omega^{i,j}$ becomes a bi-graded
algebra, and the differential splits into the sum, $\dd = \de + \deb$, of a holomorphic and antiholomorphic part, respectively 
$\de:\Omega^{\bullet,\bullet}\to\Omega^{\bullet+1,\bullet}$ and $\deb:\Omega^{\bullet,\bullet}\to\Omega^{\bullet,\bullet+1}$.

Due to the graded Leibniz rule and to the conditions $\de^2=\de\deb+\deb\de=\deb^2=0$ (that are equivalent to $\dd^2=0$),
the derivations $\de$ and $\deb$ (and hence $\dd$) are uniquely determined by their restrictions to $0$-forms,
which are described below. Finally, the $*$-structure maps $\Omega^{i,j}$ into $\Omega^{j,i}$ and the reality condition
$\dd(a^*)^*=\dd a$ is equivalent to the condition $\deb a:=(\de a^*)^*$.

\subsection{The differential calculus}\label{se:df}
As mentioned, for a general $x\in\Uq$ we have the generalized Leibniz rule in \eqref{glr} for 
the operator $\mL{x}$ when acting on $\Oq$. Using the $\Kq$-invariance of $\Aq$, for 
$x\in\{E_2,F_2,E_1E_2,F_1F_2\}$ the condition \eqref{glr} reads as  
$$
\mL{x}(ab)=(\mL{x}a)b+a(\mL{x}b) \;,\qquad\forall\;a,b\in\Aq\;,
$$
that is, we have four derivations $\Aq\to\Oq$ given by
$\mL{E_2},\mL{F_2},\mL{E_1E_2},\mL{F_1F_2}$.
 We define:
\begin{equation}\label{eq:2.5}
\de a:=\ii \, q^{-\frac{3}{2}} (a\za E_2,\, a\za E_2E_1)^t \;,\qquad
\deb a:=\ii \, (a\za F_2F_1,\, a\za F_2)^t \;,
\end{equation}
where we multiplied the operators in \cite[Rem.~5.8]{DL09b} by $-\ii$ (cf.~Remark~\ref{rem:2.1}).
It was shown in the appendix of \cite{DDL08b} that
for $q=1$ the operator $\deb$ is the usual Dolbeault
operator.

The one-forms $\de p_{ij}$, with $p_{ij}=z_i^* z_j$ the generators of $\Aq$, are a generating family for $\Omega^{1,0}$ as a one sided (left or right) module \cite[pg.~869]{DL09b}. Similarly $\de p_{ij}$ are a generating family for $\Omega^{0,1}$. An explicit computation gives:
\begin{equation}\label{eq:ddPij}
\de p_{jk}= \ii q^{-1} (u^3_j)^*\binom{u^2_k}{u^1_k} \;,\qquad
\deb p_{jk}= \ii q^{-1} \binom{q^{-\frac{1}{2}}(u^1_j)^*}{-q^{\frac{1}{2}}(u^2_j)^*}u^3_k \;.
\end{equation}
We shall also need the following formul{\ae}, that can be easily
derived from \eqref{eq:2.5} and the explicit formul{\ae} for the
right action of $\Uq$:
\begin{equation}\label{eq:ddfakePij}
\de(z_kz_l^*)=\ii q^{-2}\binom{u^2_k}{u^1_k}z_l^* \;,\qquad
\deb(z_kz_l^*)=\ii q^{-2}z_k\binom{q^{-\frac{1}{2}}(u^1_l)^*}{-q^{\frac{1}{2}}(u^2_l)^*} \;.
\end{equation}

\begin{table}[t]

\begin{center}
\begin{tabular}{ccc}
\begindc{\commdiag}[30]
 \obj(3,5)[A]{$\;\;\;\Omega^{0,0}$}
 \obj(2,4)[B]{$\;\;\;\Omega^{0,1}$}
 \obj(4,4)[C]{$\;\;\;\Omega^{1,0}$}
 \obj(1,3)[D]{$\;\;\;\Omega^{0,2}$}
 \obj(3,3)[E]{$\;\;\;\Omega^{1,1}$}
 \obj(5,3)[F]{$\;\;\;\Omega^{2,0}$}
 \obj(2,2)[G]{$\;\;\;\Omega^{1,2}$}
 \obj(4,2)[H]{$\;\;\;\Omega^{2,1}$}
 \obj(3,1)[I]{$\;\;\;\Omega^{2,2}$}
 \mor{A}{B}{}
 \mor{A}{C}{}
 \mor{B}{D}{}
 \mor{B}{E}{}
 \mor{C}{E}{}
 \mor{C}{F}{}
 \mor{D}{G}{}
 \mor{F}{H}{}
 \mor{E}{G}{}
 \mor{E}{H}{}
 \mor{G}{I}{}
 \mor{H}{I}{}
\enddc
& \raisebox{0.84in}{$=\!\!$} &
\begindc{\commdiag}[30]
 \obj(3,5)[A]{$(0,0)$}
 \obj(2,4)[B]{$(\frac{1}{2},\frac{3}{2})$}
 \obj(4,4)[C]{$(\frac{1}{2},-\frac{3}{2})$}
 \obj(1,3)[D]{$(0,3)$}
 \obj(3,3)[E]{$(1,0)\oplus (0,0)$}
 \obj(5,3)[F]{$(0,-3)$}
 \obj(2,2)[G]{$(\frac{1}{2},\frac{3}{2})$}
 \obj(4,2)[H]{$(\frac{1}{2},-\frac{3}{2})$}
 \obj(3,1)[I]{\smash[b]{$(0,0)$}}
 \mor{A}{B}{}
 \mor{A}{C}{}
 \mor{B}{D}{}
 \mor{B}{E}{}
 \mor{C}{E}{}
 \mor{C}{F}{}
 \mor{D}{G}{}
 \mor{F}{H}{}
 \mor{E}{G}{}
 \mor{E}{H}{}
 \mor{G}{I}{}
 \mor{H}{I}{}
\enddc
\end{tabular}
\end{center}

\caption{In the diamond on the right, in position $(i,j)$ we put
spin and charge $(\ell,N)$ of the representation corresponding to the
bimodule $\Omega^{i,j}$.}\label{tab}
\end{table}

We now recall how the bi-graded $*$-algebra $\Omega^{\bullet,\bullet}(\CP^2_q)$ is defined.
Each space of forms
$\Omega^{i,j}$ is defined as a bimodule associated, like in \eqref{eq:Esigma},
to a suitable representation denoted $\sigma^{i,j}:\Kq\to\mathrm{Aut}(V^{i,j})$.
The relevant representations for the present calculus are listed in Table~\ref{tab}, the only occurring ones being of type
$\sigma_{\ell,N}$ with $\ell=0,\frac{1}{2},1$, and one can check that the elements \eqref{eq:ddPij}
belong to $\Omega^{1,0}$ and $\Omega^{0,1}$ respectively. 
Note that (cf.~Table 1) the module of two-forms is the direct sum of four submodules:
\begin{equation}\label{eq:dec2forms}
\Omega^2=\Omega^{2,0}\oplus \Omega^{1,1}_v\oplus\Omega^{1,1}_s \oplus\Omega^{0,2} \;,
\end{equation}
where $\Omega^{1,1}_v$ is the module associated with the $3$-dimensional representation of $\Kq$ with spin $\ell=1$ and
charge $N=0$ and $\Omega^{1,1}_s$ is the module associated with the trivial $1$-dimensional representation. We will use
this decomposition later on.

Let $V_{\ell,N}$ be the vector space underlying the representation $\sigma_{\ell,N}$ and let
$J:V_{\ell,N}\to V_{\ell,-N}$ be the antilinear map given by
$$
Ja=a^* \;,\quad
J(v_1,v_2)^t=(-q^{-\frac{1}{2}}v_2^*,q^{\frac{1}{2}}v_1^*)^t \;,\quad
J(w_1,w_2,w_3)^t=(-q^{-1}w_3^*,w_2^*,-qw_1^*)^t \;,
$$
for any $a\in V_{0,N}$, $v\in V_{\frac{1}{2},N}$ and $w\in V_{1,N}$ respectively
(the three cases in Table~\ref{tab}). 
Then, a graded involution on $\Omega^{\bullet,\bullet}$  
is defined by
$\omega^*:=(-1)^iJ(\omega)$ for all $\omega\in\Omega^{i,j}$ \cite[Lemma~5.2]{DL09b}. 

The last ingredient we need to recall is the definition of the product.
The first step is to define a product on $V^{\bullet,\bullet}:=\bigoplus_{i,j}V^{i,j}$ as in
\cite[Prop.~5.1]{DL09b}, a result which we quote below.

\begin{prop}\label{pr:const}
A left $\Kq$-covariant graded associative product $\wprod$ on $V^{\bullet,\bullet}$,
sending real vectors to real vectors and graded commutative for $q=1$, is given by
\begin{align*}
V^{0,1}\times V^{0,1} &\to V^{0,2}\;,
   & v\wprod w &:=c_0\mu_0(v,w)^t\;,\\
V^{0,1}\times V^{1,0} &\to V^{1,1}\;,
   & v\wprod w &:=\bigl(c_1\mu_1(v,w),c_2\mu_0(v,w)\bigr)^t \;,\\
V^{0,1}\times V^{2,1} &\to V^{2,2}\;,
   & v\wprod w &:=c_3\mu_0(v,w)^t\;,\\
V^{0,1}\times V^{1,1} &\to V^{1,2}\;,
   & v\wprod w &:=\frac{c_0}{[2]c_1}\mu_2(v,w)^t-\frac{c_0}{[2]c_2}vw_4 \;,\\
V^{1,0}\times V^{1,0} &\to V^{2,0}\;,
   & v\wprod w &:=c_4\mu_0(v,w)^t\;,\\
V^{1,0}\times V^{0,1} &\to V^{1,1}\;,
   & v\wprod w &:=\bigl(-q^{\frac{1}{2}s}c_1\mu_1(v,w),q^{-\frac{3}{2}s}c_2\mu_0(v,w)\bigr)^t \;,\\
V^{1,0}\times V^{1,2} &\to V^{2,2}\;,
   & v\wprod w &:=\frac{c_3c_4}{c_0}\mu_0(v,w)^t\;,\\
V^{1,0}\times V^{1,1} &\to V^{2,1}\;,
   & v\wprod w &:=-q^{-\frac{1}{2}s}\frac{c_4}{[2]c_1}\mu_2(v,w)^t-q^{\frac{3}{2}s}\frac{c_4}{[2]c_2}vw_4 \;,\\
V^{1,2}\times V^{1,0} &\to V^{2,2}\;,
   & v\wprod w &:=\frac{c_3c_4}{c_0}\mu_0(v,w)^t\;,\\
V^{2,1}\times V^{0,1} &\to V^{2,2}\;,
   & v\wprod w &:=c_3\mu_0(v,w)^t\;,\\
V^{1,1}\times V^{0,1} &\to V^{1,2}\;,
   & v\wprod w &:=-q^{-\frac{1}{2}s}\frac{c_0}{[2]c_1}\mu_3(v,w)^t-q^{\frac{3}{2}s}\frac{c_0}{[2]c_2}v_4w \;,\\
V^{1,1}\times V^{1,0} &\to V^{2,1}\;,
   & v\wprod w &:=\frac{c_4}{[2]c_1}\mu_3(v,w)^t-\frac{c_4}{[2]c_2}v_4w \;,\\
V^{1,1}\times V^{1,1} &\to V^{2,2}\;,
   & v\wprod w&:=-q^{-\frac{1}{2}s}\frac{c_3c_4}{[2]|c_1|^2}\mu_4(v,w)-q^{\frac{3}{2}s}\frac{c_3c_4}{[2]|c_2|^2}v_4 w_4 \;,
\end{align*}
where the maps $\mu_i$'s are 
\begin{align*}
\mu_0:\R^2\times\R^2&\to\R\;, &
\mu_0(v,w)&:=[2]^{-\frac{1}{2}}(q^{\frac{1}{2}}v_1w_2-q^{-\frac{1}{2}}v_2w_1)\;,\\
\mu_1:\R^2\times\R^2&\to\R^3\;, &
\mu_1(v,w)&:=\bigl(v_1w_1,[2]^{-\frac{1}{2}}(q^{-\frac{1}{2}}v_1w_2+q^{\frac{1}{2}}v_2w_1),v_2w_2\bigr)\;,\\
\mu_2:\R^2\times\R^3&\to\R^2\;, &
\mu_2(v,w)&:=\bigl(qv_1w_2-q^{-\frac{1}{2}}[2]^{\frac{1}{2}}v_2w_1,
              q^{\frac{1}{2}}[2]^{\frac{1}{2}}v_1w_3-q^{-1}v_2w_2\bigr)\;,\\
\mu_3:\R^3\times\R^2&\to\R^2\;, &
\mu_3(v,w)&:=\bigl(q^{\frac{1}{2}}[2]^{\frac{1}{2}}v_1w_2-q^{-1}v_2w_1,
              qv_2w_2-q^{-\frac{1}{2}}[2]^{\frac{1}{2}}v_3w_1\bigr)\;,\\
\mu_4:\R^3\times\R^3&\to\R\;, &
\mu_4(v,w)&:=qv_1w_3-v_2w_2+q^{-1}v_3w_1 \;.
\end{align*}
The parameters $c_0,\ldots,c_4\in\R^\times$ and $s=\pm 1$ are not fixed for the time being.
\end{prop}

Given that $\Omega^{i,j}:=\E(\sigma^{i,j})=\Oq\boxtimes_{\sigma^{i,j}}V^{i,j}$, 
for $\omega=av\in\Omega^{i,j}$ and $\omega'=a'v'\in\Omega^{i',j'}$, with 
$a,a'\in\Oq$ and $v\in V^{i,j}$, $v'\in V^{i',j'}$, one sets
$$
\omega\wprod\omega' := (a\,a')\, (v\wprod v').
$$
>From left covariance of the product on $V^{\bullet,\bullet}$ it follows that
$\omega\wprod\omega'$ is indeed a form 
 and $\wprod$ defines a bi-graded associative product on $\Omega^{\bullet,\bullet}$.
The datum $(\Omega^{\bullet,\bullet},\wprod)$ is automatically a left $\Aq\rtimes\Uq$-module
algebra, and with the differential $\dd$ --- uniquely defined by \eqref{eq:2.5} and the Leibniz rule ---
we get a covariant differential calculus. This is a $*$-calculus if and only if $c_0=c_4$,
a condition that we assume from now on (cf.~\cite[Lemma~5.2]{DL09b}).


\subsection{Hodge duality and orientation}  
On any module $\E(\sigma)$ of type \eqref{eq:Esigma}, seen as a right module,
a Hermitian structure $(\cdot,\cdot):\E(\sigma)\times\E(\sigma)\to\Aq$ is given by Eq.~(3.1) of \cite{DL09b}:
\begin{equation}\label{eq:canHerm}
(\eta,\xi)=\sum\nolimits_{i=1}^n\eta_i^*\xi_i \;,
\end{equation}
where $\eta_i,\xi_i\in\Oq$ are the components of $\eta,\xi\in\E(\sigma)$ with respect to the canonical basis of $\C^n$.
We have in particular a Hermitian structure on each space of forms $\Omega^{i,j}$, that for $q=1$ is just the one constructed using the metric tensor.

A basis element $\vol_1$ for the rank $1$ free bimodule $\Omega^{2,2}$ is given by
the form having \mbox{non-zero} component only in degree $4$ and equal to $1$. Every other real basis element of $\Omega^{2,2}$
is of the form $\tau_\lambda=\lambda\tau_1$, with $\lambda\in\R\smallsetminus\{0\}$. We
define the Hodge star operator by
\begin{equation}\label{eq:4.4}
\bigl(\Hs\omega,\omega'\bigr)\vol_\lambda =\omega^*\wprod\omega' \;,
\end{equation}
where $(\cdot, \cdot)$ is the Hermitian structure on $\Omega^\bullet$ given by \eqref{eq:canHerm}.
In \cite{DL09b} we set $\lambda=1$. Here we start with a general $\lambda$, and argue
how this parameter is related to the orientation choice.

\begin{prop}\label{lemma:4.3}
One has $\,\Hs^2=(-1)^k\id\,$ on $k$-forms if and only if
\begin{equation}\label{eq:2.8}
|\lambda|=1 \;,\qquad
|c_1|=q^{-\frac{1}{4}s}[2]^{-\frac{1}{4}}\sqrt{|c_0|} \;,\qquad
|c_2|=q^{\frac{3}{4}s}[2]^{-\frac{1}{4}}\sqrt{|c_0|} \;,\qquad
|c_3|=[2]^{\frac{1}{2}} \;.
\end{equation}
For this choice of parameters, the spaces $\Omega^{2,0}\oplus\Omega^{0,2}$, $\Omega^{1,1}_v$ and $\Omega^{1,1}_s$
in the decomposition \eqref{eq:dec2forms} are eigenspaces of $\,\Hs\hspace{-2.5pt}$ with eigenvalues $\lambda$, $-\lambda\,\mathrm{sign}(c_0c_3)$
and $\lambda\,\mathrm{sign}(c_0c_3)$, respectively.
\end{prop}

\begin{proof}
If $\Hs \omega$ and $\omega'$ are homogeneous with different degree,
both sides of (\ref{eq:4.4}) are zero. It is then enough to
consider the case $\omega\in\Omega^{j,i}$,  
$\omega'\in\Omega^{2-i,2-j}$.  
>From the definition of the involution on forms,
for the possible values of the labels one gets
\begin{align*}
(i,j)=(0,0),(0,2),(2,0),(2,2):\quad
     \omega^*\! &\wprod \omega'=  \omega^\dag \cdot \omega' \;,\\
(i,j)=(0,1),(1,0),(1,2),(2,1):\quad
     \omega^*\! &\wprod \omega'=c_3\mu_0(\omega^*,\omega')=
     (-1)^j[2]^{-\frac{1}{2}} c_3\, \omega^\dag \cdot \omega'  \;,\\
(i,j)=(1,1):\quad
     \omega^*\! &\wprod \omega'=
     \frac{c_3c_4}{[2]}
     \bigl(-q^{-\frac{1}{2}s}|c_1|^{-2}w^\dag,q^{\frac{3}{2}s}|c_2|^{-2}w_4^\dag\bigr)
     \omega' \;.
\end{align*}
Condition \eqref{eq:4.4} is satisfied if (recall that $c_0=c_4$):
$$
(\Hs\,\omega)_{2-i,2-j}=
\begin{cases}
   \lambda^{-1}\omega_{j,i} & \mathrm{if}\;(i,j)=(0,0),(0,2),(2,0),(2,2)\;,\\[5pt]
  (-1)^j[2]^{-\frac{1}{2}}c_3\lambda^{-1}\omega_{j,i}
 & \mathrm{if}\;(i,j)=(0,1),(1,0),(1,2),(2,1)\;, \\[5pt]
 [2]^{-1}c_0c_3\lambda^{-1} (-q^{-\frac{1}{2}s}|c_1|^{-2}w,q^{\frac{3}{2}s}|c_2|^{-2}w_4)
  \hspace{-1.5cm}
 & \hspace{2cm} \mathrm{if}\;(i,j)=(1,1)\;.
\end{cases}
$$
The equation $\Hs^2\omega=(-1)^{\mathrm{dg}(\omega)}\omega$ is satisfied if the $\lambda,c_i$ are
given by \eqref{eq:2.8}. From \eqref{eq:2.8} and the equations above, the last statement easily follows.
\end{proof}

Classically (see e.g.~\cite{DK90}), on $\CP^2$ with standard orientation a two-form is ASD if and only if
it belongs to $\Omega^{1,1}_v$, while on the space $\overline{\CP}{}^2$
 a two-form is ASD if and only if its component in $\Omega^{1,1}_v$ is zero.
For $q\neq 1$, imposing \eqref{eq:2.8}, all parameters are fixed but for some arbitrary signs 
and a global rescaling encoded in $c_0$. It follows from Proposition~\ref{lemma:4.3} that, if we want
$\Omega^{2,0}\oplus\Omega^{1,1}_s\oplus\Omega^{0,2}$ to be an eigenspace of $\Hs$, as in the
classical case, we are forced to choose $c_0c_3>0$. With this choice, $\lambda=1$ gives the
standard orientation, or $\CP^2_q$, and $\lambda=-1$ the reversed orientation, or $\overline{\CP}{}^2_q$.	
We choose $\lambda=-1$.


\subsection{The K{\"a}hler form}
We define the K{\"a}hler 2-form $\qkahler $ of $\CP^2_q$ as the concrete representation, using the differential calculus, 
of a suitable element in degree two twisted cyclic homology. It is the image of the defining projection 
$p=P_{-1}$ of $\Aq$ under the equivariant
Connes-Chern character (cf.~\cite[Sec.~7]{DL09b}).
The projection $p$ is invariant, in the sense of \cite[Lemma~7.1]{DL09b},
under the representation of $\Uq$ of highest weight $(1,0)$ (cf.~\cite[Lemma~7.3]{DL09b}).
According to \cite[Theorem 7.2]{DL09b},  
we can associate to $p$ a cocycle $\mathrm{ch}^2(p,\rho^{(1,0)})\bigl(K\bigr)$ in twisted cyclic homology.
Here $K=(K_1K_2)^{-4}$ is the element implementing the square of the antipode of $\Uq$
as well as the modular automorphism (for the Haar state) on $\CP^2_q$ (cf.~\cite[Theorem~7.2]{DL09b}).
Esplicitly:
\begin{multline*}
\mathrm{ch}^2(p,\rho^{(1,0)})(K) := \tr_\rho(p^{\otimes 3+1})(K) \\
:= \sum\nolimits_{ijkl} p_{ij} \otimes p_{jk} \otimes p_{kl} \, (\rho^{(1,0)}(K))_{il}
= \sum\nolimits_{jkl}q^{6-2j}p_{jk} \otimes p_{kl} \otimes p_{lj} \;.
\end{multline*}
Using the linear representation coming from the differential calculus, we get the two-form:
$$
\qkahler :=
\sum\nolimits_{jkl}q^{6-2j}p_{jk}\dd p_{kl}\wprod \dd p_{lj} \;.
$$
As shown in Sec.~\ref{sec:AppA}, for $q=1$ the previous formula yields
the usual K{\"a}hler form of $\CP^2$ up to a normalization factor (cf. Eq.~\eqref{eq:kahlerq}).

\begin{prop}
The form $\qkahler$ is a basis of $\Omega^{1,1}_s$ as can be seen from its
explicit expression:
\begin{equation}\label{eq:Kah}
\qkahler =q^{-2}q^{-\frac{3}{2}s}c_2[2]^{\frac{1}{2}}(0,0,0,1)^t \;.
\end{equation}
Furthermore:
\begin{equation}\label{eq:useful}
p_{jk} \,\qkahler=\sum\nolimits_l\de p_{jl}\wprod \deb p_{lk} \;.
\end{equation}
\end{prop}
\begin{proof}
Using \eqref{eq:ddPij}, an explicit computation gives
\begin{align*}
\de p_{ij}\wprod \deb p_{kl}
&=-q^{-2}(u^3_i)^*\binom{u^2_j}{u^1_j}\wprod\binom{q^{-\frac{1}{2}}(u^1_k)^*}{-q^{\frac{1}{2}}(u^2_k)^*}u^3_l \\
&=-q^{-2}(u^3_i)^*\begin{bmatrix}
-q^{\frac{1}{2}(s-1)}c_1u^2_j(u^1_k)^* \\[2pt]
q^{\frac{1}{2}s}c_1[2]^{-\frac{1}{2}}\big(u^2_j(u^2_k)^*-u^1_j(u^1_k)^*\big) \\[2pt]
q^{\frac{1}{2}(s+1)}c_1u^1_j(u^2_k)^* \\[2pt]
-q^{-\frac{3}{2}s}c_2[2]^{-\frac{1}{2}}\big(qu^2_j(u^2_k)^*+q^{-1}u^1_j(u^1_k)^*\big)
\end{bmatrix}u^3_l \;.
\end{align*}
In particular, using Lemma~\ref{lemma:ortho}:
\begin{equation}\label{eq:tempUse}
\sum\nolimits_j\de p_{ij}\wprod \deb p_{jk}
=q^{-2}q^{-\frac{3}{2}s}c_2[2]^{\frac{1}{2}}p_{ik}(0,0,0,1)^t \;.
\end{equation}
Similarly,
\begin{align}
\sum\nolimits_{abc}
q^{6-2a}p_{ab}\deb p_{bc}\wprod\de p_{ca}
&=
-q^{-2}\sum\nolimits_{ab}
q^{6-2a}p_{ab}\binom{q^{-\frac{1}{2}}(u^1_b)^*}{-q^{\frac{1}{2}}(u^2_b)^*}
\wprod\binom{u^2_a}{u^1_a}=0 \notag\;,\\
\sum_c\de p_{bc}\wprod \de p_{ca} &=\sum_c\deb p_{bc}\wprod \deb p_{ca}=0 \;.
\label{eq:similarly}
\end{align}
This gives
$\qkahler =
\sum\nolimits_{abc}q^{6-2a}p_{ab}\de p_{bc}\wprod \deb p_{ca}
=q^{-2}q^{-\frac{3}{2}s}c_2[2]^{\frac{1}{2}}(0,0,0,1)^t$, where
last equality comes from \eqref{eq:tempUse} together with the $q$-trace relation $\sum_aq^{6-2a}p_{aa}=1$.
This proves \eqref{eq:Kah}.
Comparing it with \eqref{eq:tempUse} we get \eqref{eq:useful}.
\end{proof}

\noindent
As for the classical $\overline{\CP}{}^2$, the K{\"ahler} form is a basis for ASD forms
of type $(1,1)$ on $\overline{\CP}{}^2_q$.s


\subsection{Differential forms with coefficients in line bundles}\label{sec:dfwcL}
In Sec.~\eqref{se:df} the bimodules of differential forms were defined as $\Omega^{i,j}=\E(\sigma_{\ell,n})$, 
with the relevant values of $\ell,n$ given by Table \ref{tab}. 
Similarly, we shall refer to elements of the bimodule
$$
\Omega^{i,j}(L_N):=\E(\sigma_{\ell,n+N})
$$
(for the same relevant $\ell,n$) as forms with coefficients in the line bundle $L_N=\E(\sigma_{0,N})$ of Sec.~\ref{sec:line}. 
There is a bimodule isomorphism given by
$$
L_N\otimes_{\Aq}\Omega^{i,j}\to\Omega^{i,j}(L_N)
\;,\qquad
a\otimes_{\Aq}\omega\mapsto a\cdot\omega \;,
$$
with inverse
$$
\Omega^{i,j}(L_N)\to L_N\otimes_{\Aq}\Omega^{i,j}
\;,\qquad
\omega\mapsto \Psi^\dag_N\,\dot\otimes_{\Aq}\,\Psi_N\omega \;,
$$
where $\dot\otimes_{\Aq}$ the algebraic tensor product composed with row-by-column multiplication.
Similarly there is a bimodule isomorphism
$$
\Omega^{i,j}\otimes_{\Aq}L_N\to\Omega^{i,j}(L_N)
\;,\qquad
\omega\otimes_{\Aq} a\mapsto \omega\cdot a \;,
$$
with inverse
$$
\Omega^{i,j}(L_N)\to\Omega^{i,j}\otimes_{\Aq}L_N
\;,\qquad
\omega\mapsto \omega\Psi^\dag_N\,\dot\otimes_{\Aq}\,\Psi_N \;.
$$
As a consequence, $\Omega^k\otimes_{\Aq}L_N$ and $\Omega^k\otimes_{\Aq}L_N$ are isomorphic as bimodules.
The former space would be the source and target of left-module connections on $L_N$,
while the latter would be the source and target of right-module connections.
Due to the above bimodule isomorphisms, we shall think of both left-module and right-module connections on $L_N$ as maps
from $\Omega^k(L_N)$ to $\Omega^{k+1}(L_N)$.


\section{Instantons on $\CP^2_q$}\label{sec:3}
Classically, one-instantons on $\CP^2$ are connections on the vector bundle
associated with a $\SU(2)$-principal bundle $\mathcal{P}\to\CP^2$ via the
fundamental representation of $\SU(2)$. The total space of such a principal bundle is
$\mathcal{P}=\mathrm{S}^5\times_{\mathrm{U}(1)}\SU(2)$ where $u \in \mathrm{U}(1)$ acts on the coordinates $z_i$'s of 
$\mathrm{S}^5$ by multiplication and the embedding $\mathrm{U}(1)\to \SU(2)$ is given by
$$
u \mapsto\maa{ u & 0 \\ 0 & \bar u } \;.
$$

The vector bundle associated with the fundamental representation of $\SU(2)$ is then
$\mathcal{P}\times_{\SU(2)}\C^2\simeq \mathrm{S}^5\times_{\mathrm{U}(1)}\C^2$ and the representation
of $\mathrm{U}(1)$ on $\C^2$ is the sum of the fundamental one and its dual. The resulting vector
bundle is the direct sum of the tautological bundle with its dual,
and the corresponding module of sections is $L_1\oplus L_{-1}$.
This is the module where we will construct one-instantons for $q\neq 1$.


\subsection{The reducible instanton}

Using the module isomorphism \eqref{eq:isoPsiN} one trasports
on $L_N$ the Grassmannian connection $P_N \circ \dd\,$ of $P_N\Aq^{\rPN}$, with $P_N$ the corresponding projection in \eqref{mon-pro}. 
The result is a
right-module connection $\nabla_{\!N}:\Omega^k(L_N)\to\Omega^{k+1}(L_N)$
given, for all $\eta\in\Omega^k(L_N)$, by
\begin{equation}\label{eq:conN}
\nabla_{\!N}\eta=\Psi_N^\dag\dd(\Psi_N\eta) \;.
\end{equation}
 On the other hand, viewing $L_N$ as a left module, and using the isomorphism \eqref{eq:isoPsiNa} to trasport
on it the Grassmannian connection of $\Aq^{\rPN}P_{-N}$, one gets
the left module connection $\nabla_{\!N}^\ell$ given, for all $\eta\in\Omega^k(L_N)$ by 
\begin{equation}\label{eq:conN-left}
\nabla_{\!N}^\ell\eta=\dd(\eta\Psi_{-N}^\dag)\,\Psi_{-N} \;.
\end{equation}

Being $\nabla_N$ a right-module connection,  
it satisfies by construction a `right' Leibniz rule:
$$
\nabla_N(\eta a)=\nabla_N(\eta)a+\eta\,\dd a \;,
$$
for any $a\in\Aq$ and $\eta\in L_N$, while $\nabla_{\!N}^\ell$ satisfies a left version of the above. However, 
since $L_N$ is a bimodule,  there is also a `left' Leibniz rule for $\nabla_N$, that is for a product $a\eta$ (and similarly 
a `right' Leibniz rule for $\nabla_{\!N}^\ell$, that is for a product $\eta a$).
\begin{prop}
For any $a\in\Aq$ and $\eta\in L_N$: 
\begin{equation}\label{eq:leftLeib}
\nabla_{\!N}(a\eta)
=a\nabla_{\!N}(\eta)+q^N(\dd a)\eta \;.
\end{equation}
\end{prop}

\begin{proof}
>From (6.5) and (6.6) of \cite{DL09b} and the analogous relations for $N<0$ one has
\begin{equation}\label{eq:follows}
\Psi_N^\dag(\Psi_N\za E_2)=
\Psi_N^\dag(\Psi_N\za F_2)=
(\Psi_N^\dag\za E_2)\Psi_N=
(\Psi_N^\dag\za F_2)\Psi_N=
0 \;,
\end{equation}
for any $N\in\Z$.
Hence \eqref{eq:2.5} and the right $\U_q(\mathfrak{su}(2))$-invariance 
of $\Sq$, yield
$$
\Psi^\dag_N\dd(\Psi_Na\Psi_N^\dag)\Psi_N=
\Psi^\dag_N(\Psi_N\za K_2^{-1})(\dd a)(\Psi_N\za K_2^{-1})^\dag\Psi_N
\;,
$$
which gives $q^N\dd a$ after using (6.4) of \cite{DL09b}, that is $\Psi_N\za K_2=q^{-\frac{N}{2}}\Psi_N$.
This proves that, 
\begin{equation}\label{eq:ddPsi}
\Psi^\dag_N\dd(\Psi_Na\Psi_N^\dag)\Psi_N=q^N\dd a \;,
\end{equation}
for any $a\in\Aq$. 
>From the definition \eqref{eq:conN} and using $\Psi_N^\dag\Psi_N=1$ one gets
$$
\nabla_{\!N}(a\eta)=\Psi_N^\dag\dd(\Psi_Na\Psi_N^\dag\Psi_N\eta) \;,
$$
and, using the Leibniz rule for $\dd$,  
$$
\nabla_{\!N}(a\eta)=a\Psi_N^\dag\dd(\Psi_N\eta)
+\Psi_N^\dag\dd(\Psi_Na\Psi_N^\dag)\Psi_N\eta
=a\nabla_{\!N}(\eta)+\Psi_N^\dag\dd(\Psi_Na\Psi_N^\dag)\Psi_N\eta \;.
$$
Equation \eqref{eq:leftLeib} then follows from \eqref{eq:ddPsi}.
\end{proof}

The anti-linear map $\eta\to\eta^*$ sends $L_N$ to $L_{-N}$.  This is not a bimodule isomorphism;
nevertheless one easily finds that
\begin{equation}\label{left-right}
(\nabla_{\!N}\eta)^*=\nabla_{\!-N}^\ell(\eta^*) \;.
\end{equation}
More generally, the space of connections being affine, 
any connection on $L_N$ is obtained from $\nabla_N$ (or $\nabla_{\!N}^\ell$)
by adding a right (or left) module endomorphism of $L_N$ with coefficients in one-forms, that is an element in $\Omega^1$, acting by multiplication
from the left (or from the right, respectively).  Thus, any right-module connection
on $L_N$ is of the type $\nabla_N+\omega\wprod(\,.\,)$ and any left-module connection on
$L_{-N}$ is of the type $\nabla_N+(\,.\,)\wprod\omega$ with $\omega\in\Omega^1$;
with a slight abuse of terminology, we shall refer to $\omega$, in both cases, as the \emph{connection one-form}.
Since the projections $P_N$ are self-adjoint, one can check that the connection
$\nabla_{\!N}$ is compatible with the Hermitian structure \eqref{eq:canHerm}, this meaning
$(\nabla_N\eta,\xi)+(\eta,\nabla_N\xi)=\dd(\eta,\xi)$. As a consequence, a
connection $\nabla_N+\omega$ is Hermitian if and only if $(\omega\eta,\xi)+(\eta,\omega\xi)=0$, that is
$\omega=-\omega^*$.
Then, for Hermitian connections it follows that
\begin{equation}\label{con-left-right}
(\nabla_{\!N}\eta+\omega\wprod\eta)^*=\nabla_{\!-N}^\ell(\eta^*)+\eta^*\wprod\omega \;.
\end{equation}
Thus, conjugation trasform a right-module Hermitian connection into a left module Hermitian connection with the same connection one-form.

\medskip

The curvature of the connection in \eqref{eq:conN} is the operator of multiplication \emph{from the left} by the
(scalar) two-form $\F_N$ given by~\cite[eqn.~(6.3)]{DL09b}:
\begin{equation}\label{eq:FN}
\F_N=\Psi_N^\dag(\dd P_N\wprod \dd P_N)\Psi_N \;.
\end{equation}
On the other hand, the curvature of the left module connection $\nabla_{\!N}^\ell$ is the operator of multiplication
\emph{from the right} by the two-form $\F_{-N}$ given by the same formula above,
but with $N$ replaced by $-N$. We will work with right modules from now on,
but we stress that trading right modules (and connections) for left ones
simply amounts to changing sign to the label $N$.

\begin{prop}
The curvature two-form $\F_N$ is proportional to the K{\"a}hler form, hence it is ASD,
for any $N\in\Z$. More precisely
$$
\F_N=
\begin{cases}
q^{N-1}[N]\,\omega_q & \forall\;N\geq 0, \\[4pt]
q^{N+\frac{3}{2}s+2}[N]\,\qkahler & \forall\;N<0.
\end{cases}
$$
\end{prop}
\begin{proof}
The proportionality constant has been computed in \cite{DL09b}
for $N\geq 0$ (cf.~eqn.~(6.8) there), resulting into $\F_N=q^{N-1}[N]\F_1$.

Since $\Psi_1^\dag=(z_1,z_2,z_3)$ and $\sum_jz_j\deb p_{ij}=\sum_j(\de p_{ij})z_j=0$, it follows that
$$
\F_1=\sum_{abc}z_a(\dd p_{ab}\wprod\dd p_{bc})z_c^*
    =\sum_{abc}z_a(\de p_{ab}\wprod\deb p_{bc})z_c^*
$$
and from \eqref{eq:useful} we get $\F_1=\omega_q$. This proves the statement for $N\geq 0$.

For negative $N$, with the same proof after Lemma 6.2 of \cite{DL09b} one gets
$\F_N=-q^{N+1}[N]\F_{-1}$. Since $\Psi_{-1}=(q^2z_1,qz_2,z_3)^t$, it follows that
$$
\F_{-1}=\sum_{abc}q^{6-2a}q^{6-2b}q^{6-2c}z_a^*\dd (z_az_b^*)\wprod\dd (z_bz_c^*)z_c \;.
$$
In turn, from \eqref{eq:ddfakePij} and the orthogonality relations in Lemma~\ref{lemma:ortho}:
$$
\sum\nolimits_aq^{6-2a}z_a^*\de (z_az_b^*)=\sum\nolimits_cq^{6-2c}\deb (z_bz_c^*)z_c=0 \:, 
$$
and
\begin{multline*}
\sum_{abc}q^{6-2a}q^{6-2b}q^{6-2c}z_a^*\deb (z_az_b^*)\wprod\de (z_bz_c^*)z_c
=-q^{-4}\sum\nolimits_bq^{6-2b}\binom{q^{-\frac{1}{2}}(u^1_b)^*}{-q^{\frac{1}{2}}(u^2_b)^*}\wprod
\binom{u^2_b}{u^1_b} \\
=-q^{-4}\sum\nolimits_bq^{6-2b}
\begin{bmatrix}
c_1q^{-\frac{1}{2}}(u^1_b)^*u^2_b \\[2pt]
c_1[2]^{-\frac{1}{2}}\big(q^{-1}(u^1_b)^*u^1_b-q(u^2_b)^*u^2_b\big) \\[2pt]
-c_1q^{\frac{1}{2}}(u^2_b)^*u^1_b \\[2pt]
c_2[2]^{-\frac{1}{2}}\big((u^1_b)^*u^1_b+(u^2_b)^*u^2_b\big)
\end{bmatrix} 
=-q^{-4}\sum\nolimits_b
\begin{bmatrix}
0 \\
0 \\
0 \\
c_2q^3[2]^{\frac{1}{2}}
\end{bmatrix}=-q^{\frac{3}{2}s+1}\qkahler \;.
\end{multline*}
Hence $\F_{-1}=-q^{\frac{3}{2}s+1}\qkahler$ and $\F_N=q^{N+\frac{3}{2}s+2}[N]\,\qkahler$ for all $N<0$.
\end{proof}

As a consequence of previous proposition, the direct sum
\begin{equation}\label{eq:1on1}
\widetilde{\nabla}_0:=\nabla_1\oplus\nabla_{-1}
\end{equation}
is a reducible ASD connection on $L_1\oplus L_{-1}$. The reason to call $\widetilde{\nabla}_0$ an instanton comes from the already mentioned fact that the module $L_1\oplus L_{-1}$ has the correct `topological numbers' , i.e rank $2$, charge $0$, and instanton number $1$, respectively.

Next subsections are devoted to construct irreducible instantons on $L_1\oplus L_{-1}$.


\subsection{A family of instantons}\label{se:1-pfi}
The space of connections (on a fixed right module) being affine, 
any connection on $L_1\oplus L_{-1}$ can be obtained from $\widetilde{\nabla}_0$
by adding a right module endomorphism with coefficients in $\Omega^1$.
 Module maps $P_n\Aq^{r_n}\to P_m\Aq^{r_m}$ are given
by left multiplication by elements of $P_mM_{r_m\times r_n}(\Aq)P_n$,
so right module maps $L_n\to L_m$ are given by
elements $\Psi_m^\dag T\Psi_n$ with $T\in M_{r_m\times r_n}(\Aq)$.
In particular, recalling that $\Psi_1^\dag=(z_1,z_2,z_3)$
and $\Psi_{-1}=(q^2z_1,qz_2,z_3)^t$,
we see that any right module map $L_{-1}\to L_1$ is a linear combination
of elements $z_iz_j$ with coefficients in $\Aq$; but elements $z_iz_j$ are a
generating family for the module $L_2$, hence the identification
$$
\mathrm{Hom}_{\Aq}(L_{-1},L_1)\simeq L_2 \;.
$$
Therefore, the most general Hermitian connection on $L_1\oplus L_{-1}$ can be written as
\begin{equation}\label{eq:ABPhi}
\widetilde{\nabla}=\widetilde{\nabla}_0+
\begin{pmatrix}
\;A^*-A & \Phi\; \\
\;-\,\Phi^* & B-B^*
\end{pmatrix}
\end{equation}
with $A,B\in\Omega^{0,1}$ and $\Phi\in \Omega^1(L_2)$. 

\begin{rem}\label{rem:ac}
As mentioned, the Grassmannian connection (for a self-adjoint projection) is compatible with the Hermitian structure
\eqref{eq:canHerm}, thus $(\widetilde{\nabla}_0\eta,\xi)+(\eta,\widetilde{\nabla}_0\xi)=\dd(\eta,\xi)$.
If $\widetilde{\nabla}$ is also Hermitian, the difference $\omega:=\widetilde{\nabla}-\widetilde{\nabla}_0$
satisfies $(\omega\eta,\xi)+(\eta,\omega\xi)=0$. From the definition \eqref{eq:canHerm}, it follows
that $\omega$ has to be an antihermitian matrix, as in \eqref{eq:ABPhi}.
\end{rem}

In order to proceed, at this point for the $q=1$ case one makes an ansatz \cite{Gro90} on the elements in the $2\times 2$ 
matrix in \eqref{eq:ABPhi}. One possibility could be to make the analogous ansatz here: we could assume that
the element $\Phi$ in \eqref{eq:ABPhi} is proportional to  
\begin{equation}\label{eq:phi}
\phi:=q^{\frac{1}{2}}\sum\nolimits_j\bigl\{z_j(\de p_{j2})z_3-qz_j(\de p_{j3})z_2\bigr\} \:,
\end{equation}
that is $\Phi=C\phi$ with $C\in\Aq$.
Moreover, we could assume that the element $C$ and the 1-forms $A,B$ in \eqref{eq:ABPhi}
are ``functions'' of $p_{11}$ alone:
$$
C = C(p_{11}) \;, \quad A = A(p_{11}) \;, \quad B = B(p_{11}) \;.
$$ 
With these assumptions, the ASD condition for the connection $\widetilde{\nabla}$ is equivalent to a system of $q$-difference equations. These are quadratic in $A$ and $B$ and linear in their $q$-derivatives (as defined below in \eqref{q-der}). 
When $q=1$, the quadratic terms disappear and $q$-derivatives become ordinary derivatives. If $q\neq 1$,
the presence of $q$-derivatives and of quadratic terms makes the problem very difficult to solve. 
To overcome these difficulties, in the next section we follow a different approach: we will construct
instantons more in the spirit of the ADHM construction.  
We will prove that the connection is of the form 
\eqref{eq:ABPhi} (cf.~Prop.~\ref{eq:prop}) with matrix elements that turn out to be of the kind of the ansatz above.

\medskip
The link between self-duality and $q$-difference equations is in the following lemmas,
that we collect here since they will also be needed later on.

\begin{lemma}
Let $\,x:=p_{11}=z_1^*z_1$. Then
\begin{subequations}\label{eq:lemma36}
\begin{gather}
(\de x)x=q^2x(\de x) \;,\qquad\quad
(\deb x)x=q^{-2}x(\deb x) \;, \label{eq:lemma36A} \\
\intertext{and, for any $n\geq 1$,}
\de x^n=[n]_q(qx)^{n-1}\de x \;,\qquad
\deb x^n=[n]_q(q^{-1}x)^{n-1}\deb x \;. \label{eq:lemma36B}
\end{gather}
\end{subequations}
\end{lemma}
\begin{proof}
As a particular case of \eqref{eq:ddPij}:
$$
\de x= \ii q^{-1} z_1^*\binom{u^2_1}{u^1_1} \;.
$$
>From the defining relations of $\Oq$ one gets $u^i_1z_1=u^i_1u^3_1=qu^i_1u^3_1=qz_1u^i_1$ for $i<3$.
Since $z_1^*=-q^{-2}(u^1_2u^2_3-qu^1_3u^2_2)$, one also checks that $u^i_1z_1^*=qu^i_1z_1^*$
for $i<3$. This proves the first equation in \eqref{eq:lemma36A}. The
second follows by conjugation.

Using \eqref{eq:lemma36A} and the Leibniz rule:
$$
\de x^n=\sum\nolimits_{k=0}^{n-1}x^{n-1-k}(\de x)x^k
=\left\{\sum\nolimits_{k=0}^{n-1}q^{2k}\right\}x^{n-1}(\de x)
$$
and $\sum\nolimits_{k=0}^{n-1}q^{2k}=(1-q^{2n})(1-q^2)^{-1}
=q^{n-1}[n]_q$, proving the first equation in \eqref{eq:lemma36B}.
One proves the second one in a similar fashion.
\end{proof}

As a corollary of previous lemma, we have the following interesting relation
between our differential calculus and the well known $q$-derivative.
The latter will be denoted simply with a `dot', instead of the more common
notation $D_q$.

\begin{lemma}\label{cor:qder}
Let the $q$-derivative be defined by
\begin{equation}\label{q-der}
\dot{f}(x):=\frac{f(qx)-f(q^{-1}x)}{(q-q^{-1})x} \;.
\end{equation}
Then
\begin{equation}\label{q-derB}
\de f(x)=\dot{f}(qx)\,\de x=\de x\,\dot{f}(q^{-1}x) \;, \qquad
\deb f(x)=\dot{f}(q^{-1}x)\,\deb x=\deb x\,\dot{f}(qx) \;.
\end{equation}
\end{lemma}

\noindent%
Previous lemma is true for $f(x)$ a polynomial of $x$, but more generally holds for any $f$ for which the $q$-derivative exists.

In fact, in the rest of this section, we need to consider a (slightly) enlarged algebra than $\Aq$.
>From the spherical relation \eqref{eq:sphere}, we deduce that
$||p_{11}||\leq ||z_1||^2\leq 1$ in any bounded $*$-representation; thus any sum $f:=\sum_{n\geq 0}c_n(p_{11})^n$ with rapid decay
coefficients $\{c_n\}$ is convergent in the universal $C^*$-algebra
$C(\CP^2_q)$ generated by the $\{p_{ij}\}$.
If $C(\SU_q(3))$ is the universal $C^*$-algebra generated by the $\{u^i_j\}$,
we can set
 $\de f:=\sum_{n\geq 0}c_n\de(p_{11})^n$ and
$\deb f:=\sum_{n\geq 0}c_n\deb(p_{11})^n$: these sums converge to some
elements of $C(\SU_q(3))^2$ that we take by definition as derivatives of $f$,
and satisfy \eqref{q-derB}.
In particular we shall need the element
\begin{equation}\label{eq:fx}
(1-t^2q^kp_{11})^{-1}=\sum\nolimits_{n\geq 0}(t^2q^kp_{11})^n \;,
\end{equation}
which is a well defined positive operator for any $0\leq t<1$
and $k\geq 0$, as well as its positive square root. We also recall that the coproduct of $\A(\SU_q(3))$ extends to
a $C^*$-algebra morphism, needed later on when constructing additional solutions out of \eqref{t-con}.

As a last remark, we note that if $A\subset B$ are $*$-algebras
and $M_A$ is a right $A$-module, this can be canonicaly turned into a right
$B$-module $M_B$ with the formula $M_B:=M_A\otimes_AB$. With a slight
abuse of notations, we will not introduce new symbols for the algebras
and modules enlarged with the square root of the element \eqref{eq:fx}
and its positive powers.

\medskip
 
We are ready to construct a family of one-instantons $\widetilde{\nabla}_t$ parametrized by $0\leq t<1$. Motivated by the classical case, discussed in 
 Sec.~\ref{sec:AppA}, we look for a projection $\widetilde{\Psi}\widetilde{\Psi}^\dag$ which is a deformation of
 the classical matrix in Eq.~\eqref{eq:decon}. 
Let $0\leq\theta<\pi/4$, $t:=\sin 2\theta$ and define
\begin{equation}\label{pf}
\psi:=
\begin{pmatrix}
\cos 2\theta\,z_1^* \\
\cos\theta\,z_2^* \\
\cos\theta\,z_3^* \\
0 \\
\sin\theta\,z_3^* \\
\sin\theta\,z_2^*
\end{pmatrix}
\;,\qquad
\varphi:=
\begin{pmatrix}
0 \\
-\sin\theta\,z_3 \\
q\sin\theta\, z_2 \\
q^2\cos 2\theta\, z_1 \\
q\cos\theta\, z_2 \\
-\cos\theta\,z_3
\end{pmatrix} \;.
\end{equation}
Entries of $\psi$ are easily seen to be a generating family for $L_1$ while entries of $\varphi$ are a generating family for 
$L_{-1}$. The $q$ factors are inserted so that  
\begin{equation}\label{eq:psiphi}
\psi^\dag\varphi=0 \;,\qquad
\psi^\dag\psi=1-t^2p_{11} \;,\qquad
\varphi^\dag\varphi=1-t^2q^4p_{11} \;;
\end{equation}
 then, if we form the $2\times 6$  matrix
\begin{equation}\label{eq:Psitilde}
\widetilde{\Psi}=
(\psi , 0)\frac{1}{\sqrt{1-t^2p_{11}}}
+
(0,\varphi)
\frac{1}{\sqrt{1-t^2q^4p_{11}}} \;,
\end{equation}
we have $\widetilde{\Psi}^\dag\widetilde{\Psi}=1_2$ and $P:=\widetilde{\Psi}\widetilde{\Psi}^\dag$
is a projection and we can consider the corresponding Grassmannian connection. One of the main results of present 
paper is the following:

\begin{thm}\label{thm} For any $0\leq t<1$, the connection on $L_1\oplus L_{-1}$ defined by
\begin{equation}\label{t-con}
\widetilde{\nabla}_t\eta:=\widetilde{\Psi}^\dag\dd(\widetilde{\Psi}\eta) \;, 
\end{equation}
has ASD curvature 
\begin{equation}\label{t-cur}
\F=\widetilde{\Psi}^\dag(\dd P \wprod \dd P)\widetilde{\Psi} \;.
\end{equation}
\end{thm}

\noindent
Note that $\widetilde{\Psi}|_{t=0}=\Psi_1\oplus\Psi_{-1}$ and $\widetilde{\nabla}_0:=\nabla_1\oplus\nabla_{-1}$ is
the reducible connection in \eqref{eq:1on1}.

\begin{proof}
With our choice of orientation, $(2,0)$ and $(0,2)$ forms are ASD; thus we only need to compute the $(1,1)$ component of the curvature $F$ and check that it is proportional to the K{\"a}hler form. Let $\F_{ij}$ be the matrix elements of the $\Omega^{1,1}$-component of $\F$, and $\F_{ij}^+$ the self-dual part. Since $\F_{21}=\F_{12}^*$, we only need to compute three components.

The details of the meticulously crafted, very long and very technical proof are relegated to Appendix~\ref{app:B}.
One finds indeed $\F_{12}=0$; this is shown in Lemma~\ref{lem:12}. That $\F_{11}^+=0$ and $\F_{22}^+=0$ is proved in Corollary~\ref{cor:6} and Corollary~\ref{cor:22+} respectively.
\end{proof}

In the next section we shall obtain new instantons from the connection \eqref{t-con} by using the coaction of $\Oq$.
For this, we need the connection one-form of \eqref{t-con} that we work out in the form \eqref{eq:ABPhi}. Its entries
turn out to be just like the ones in the anstaz mentioned before, that is $\Phi=C\phi$, with $\phi$ as in in \eqref{eq:phi}
and $A,B,C$ functions of $x=p_{11}$.

Once more, the space of connections being affine, one can decompose the connection as
\begin{equation}\label{c1f}
\widetilde{\nabla}_t=\widetilde{\nabla}_0+\omega_t
\end{equation}
with $\omega_t$ an endomorphism of $L_1\oplus L_{-1}$ with coefficients in $\Omega^1$.
Indeed, from Remark~\ref{rem:ac} $\widetilde{\nabla}_t$ is of the form
\eqref{eq:ABPhi}, for some $A,B,\Phi$ which we compute in this section.

\begin{prop}\label{eq:prop}
The matrix of 1-forms $\omega_t$ in the connection \eqref{c1f} is as in \eqref{eq:ABPhi} with
\begin{subequations}
\begin{align}
A^* &=q\sqrt{1-t^2x}\;\de
\frac{1}{\sqrt{1-t^2x}}
\;,\label{eq:prop1}\\
B &=q^{-1}\sqrt{1-t^2q^4x}\;\deb
\frac{1}{\sqrt{1-t^2q^4x}}
\;,\label{eq:prop2}\\
\Phi &=q^{-\frac{1}{2}}t\,
\frac{1}{\sqrt{1-t^2x}}\,\phi
\,\frac{1}{\sqrt{1-t^2q^4x}}
\;,\label{eq:prop3}
\end{align}
\end{subequations}
where $\phi=q^{\frac{1}{2}}\sum\nolimits_j\bigl\{z_j(\de p_{j2})z_3-qz_j(\de p_{j3})z_2\bigr\}$ as given in \eqref{eq:phi} and $x:=p_{11}$ as before.
\end{prop}

We need some preliminary results.
\begin{lemma}\label{prop:3.4}
If $\phi$ is the element in \eqref{eq:phi}, then
\begin{equation}\label{eq:phip33}
\phi= \ii \dbinom{q^{-\frac{1}{2}}(u^1_1)^*}{-q^{\frac{1}{2}}(u^2_1)^*} \;,\quad\qquad
\phi^*= \ii \dbinom{u^2_1}{u^1_1} \;.
\end{equation}
\end{lemma}
\begin{proof}
>From \eqref{eq:ddPij} and \eqref{eq:sphere} we get
$$
\phi= \ii q^{-\frac{1}{2}}\binom{u^2_2}{u^1_2}z_3- \ii q^{\frac{1}{2}}\binom{u^2_3}{u^1_3}z_2 \;,
$$
and using \eqref{eq:star}:
$$
\phi= \ii q^{-\frac{1}{2}}\binom{u^2_2u^3_3-qu^2_3u^3_2}{u^1_2u^3_3-qu^1_3u^3_2}=
\ii \binom{q^{-\frac{1}{2}}(u^1_1)^*}{-q^{\frac{1}{2}}(u^2_1)^*} \;,
$$
and this proves the first equation in \eqref{eq:phip33}. The second follows by conjugation.
\end{proof}

\begin{lemma}\label{lemma:3.1}
For any $\eta\in L_1$ and $\eta'\in L_{-1}$ it holds that
$$
z_1\dd (z_1^*\eta)
=p_{11}\nabla_1\eta
+q(\deb p_{11})\eta
\;,\qquad
z_1^*\dd (z_1\eta')=
p_{11}\nabla_{-1}\eta'
+q^{-1}(\de p_{11})\eta'
\;.
$$
\end{lemma}
\begin{proof}
Recall that $z_i\za F_2=z_i^*\za E_2=0$,
$z_i\za K_2=q^{\frac{1}{2}}z_i$ and that
$z_i$ is in the kernel of $E_1,F_1,K_1-1$. 
>From \eqref{eq:2.5}, using the coproduct and the above
observations, we get:
\begin{align*}
z_1\dd (z_1^*\eta) &=z_1(z_1^*\za K_2^{-1})\dd\eta+
z_1(\deb z_1^*)(\eta\za K_2)
=q^{\frac{1}{2}}p_{11}\dd\eta+
q^{\frac{1}{2}}z_1(\deb z_1^*)\eta
\;,\\
z_1^*\dd (z_1\eta') &=z_1^*(z_1\za K_2^{-1})\dd\eta'+
z_1^*(\de z_1)(\eta'\za K_2)
=q^{-\frac{1}{2}}p_{11}\dd\eta'+q^{-\frac{1}{2}}z_1^*(\de z_1)\eta'
\;,\\
q\deb p_{11} &=q(z_1\za K_2^{-1})\deb z_1^*
=q^{\frac{1}{2}}z_1\deb z_1^*
\;,\\
q^{-1}\de p_{11} &=q^{-1}(z_1^*\za K_2^{-1})\de z_1
=q^{-\frac{1}{2}}z_1^*\de z_1
\;.
\end{align*}
Thus
$$
z_1\dd (z_1^*\eta)=q^{\frac{1}{2}}p_{11}\dd\eta+q(\deb p_{11})\eta \;,\qquad
z_1^*\dd (z_1\eta')=q^{-\frac{1}{2}}p_{11}\dd\eta'+q^{-1}(\de p_{11})\eta' \;.
$$
>From (6.4), (6.5) and (6.6) of \cite{DL09b} and
the definition \eqref{eq:conN},
we see that
$q^{\frac{1}{2}}\dd\eta=\nabla_1\eta$.
In a similar way one shows that
$q^{-\frac{1}{2}}\dd\eta'=\nabla_{-1}\eta'$. This concludes the
proof.
\end{proof}

\begin{proof}[Proof of Proposition~\ref{eq:prop}]
Now, for $\eta=(\eta_1,0)^t$ and $\eta'=(0,\eta_2)^t$ one has
\begin{subequations}
\begin{align}
(\widetilde{\nabla}_t\eta)_1&=
\frac{1}{\sqrt{1-t^2p_{11}}}\,\psi^\dag\,
\dd
\left(
\psi\,\frac{1}{\sqrt{1-t^2p_{11}}}\,\eta_1
\right) \label{eq:omega11}
\;,\\
(\widetilde{\nabla}_t\eta)_2&=
\frac{1}{\sqrt{1-t^2q^4p_{11}}}\,\varphi^\dag\,
\dd
\left(
\psi\,\frac{1}{\sqrt{1-t^2p_{11}}}\,\eta_1
\right) \label{eq:omega21}
\;,\\
(\widetilde{\nabla}_t\eta')_2&=
\frac{1}{\sqrt{1-t^2q^4p_{11}}}\,\varphi^\dag\,
\dd
\left(
\varphi\,\frac{1}{\sqrt{1-t^2q^4p_{11}}}\,\eta_2
\right) \label{eq:omega22}
\;.
\end{align}
\end{subequations}
Notice that for any $\xi_1\in L_1$, $\xi_2\in L_{-1}$, and denoting $\xi=(\xi_1,0)$ and $\xi'=(0,\xi_2)$,
an easy algebraic manipulation gives:
\begin{align*}
\psi^\dag\dd(\psi\xi_1) &=(\widetilde{\nabla}_0\xi)_1
-t^2z_1\dd (z_1^*\xi_1)
 \;,
\\
\varphi^\dag\dd(\varphi\xi_2) &=(\widetilde{\nabla}_0\xi')_2
-t^2q^4z_1^*\dd (z_1\xi_2)
 \;,
\end{align*}
and using Lemma~\ref{lemma:3.1} we get:
\begin{align*}
\psi^\dag\dd(\psi\xi_1) &=(1-t^2p_{11})(\widetilde{\nabla}_0\xi)_1
-t^2q(\deb p_{11})\xi_1
\\
&=\psi^\dag\psi(\widetilde{\nabla}_0\xi)_1
-t^2 q(\deb p_{11})\xi_1
 \;,
\\
\varphi^\dag\dd(\varphi\xi_2) &=(1-t^2q^4p_{11})(\widetilde{\nabla}_0\xi')_2
-t^2q^3(\de p_{11})\xi_2
\\
&=\varphi^\dag\varphi(\widetilde{\nabla}_0\xi')_2
-t^2q^3(\de p_{11})\xi_2
 \;.
\end{align*}
Using these two equations in \eqref{eq:omega11} and \eqref{eq:omega22},
and inserting the identity matrix in the form $(\psi^\dag\psi)^{-1}\psi^\dag\psi$ in
\eqref{eq:omega21}, one gets
\begin{align*}
(\widetilde{\nabla}_t\eta)_1&=
\sqrt{1-t^2p_{11}}\;\widetilde{\nabla}_0\!
\left(
\frac{1}{\sqrt{1-t^2p_{11}}}\,\eta
\right)_1
-t^2 q
\frac{1}{\sqrt{1-t^2p_{11}}}\,
(\deb p_{11})\,\frac{1}{\sqrt{1-t^2p_{11}}}\,\eta_1
\;,\\
(\widetilde{\nabla}_t\eta)_2&=
\frac{1}{\sqrt{1-t^2q^4p_{11}}}\,\varphi^\dag\,
\dd
\left(
\psi\psi^\dag\psi\,\frac{1}{(1-t^2p_{11})^{\frac{3}{2}}}\,\eta_1
\right)
\;,\\
(\widetilde{\nabla}_t\eta')_2&=
\sqrt{1-t^2q^4p_{11}}\;\widetilde{\nabla}_0\!
\left(
\frac{1}{\sqrt{1-t^2q^4p_{11}}}\,\eta'
\right)_2
-t^2q^3
\frac{1}{\sqrt{1-t^2q^4p_{11}}}\,
(\de p_{11})
\,\frac{1}{\sqrt{1-t^2q^4p_{11}}}\,\eta_2
\;.
\end{align*}
Using the Leibniz rule and \eqref{eq:psiphi}
we can rewrite the second as
$$
(\widetilde{\nabla}_t\eta)_2=
\frac{1}{\sqrt{1-t^2q^4p_{11}}}\,\varphi^\dag\,
\dd(\psi\psi^\dag)\psi\,\frac{1}{(1-t^2p_{11})^{\frac{3}{2}}}\,\eta_1 \;.
$$
Therefore, using \eqref{eq:leftLeib}
for $\widetilde{\nabla}_0=\nabla_{1}\oplus\nabla_{-1}$,
we find $\widetilde{\nabla}_t=\widetilde{\nabla}_0+\omega_t$, with
\begin{align*}
(\omega_t)_{11} &=q
\sqrt{1-t^2p_{11}}\;\dd
\frac{1}{\sqrt{1-t^2p_{11}}}
-t^2 q
\frac{1}{\sqrt{1-t^2p_{11}}}\,
(\deb p_{11})\,\frac{1}{\sqrt{1-t^2p_{11}}}
\;,\\[5pt]
(\omega_t)_{21} &=
\frac{1}{\sqrt{1-t^2q^4p_{11}}}\,\varphi^\dag\,
\dd(\psi\psi^\dag)\psi\,\frac{1}{(1-t^2p_{11})^{\frac{3}{2}}}
\;,\\[5pt]
(\omega_t)_{22} &=q^{-1}
\sqrt{1-t^2q^4p_{11}}\;\dd
\frac{1}{\sqrt{1-t^2q^4p_{11}}}
-t^2q^3
\frac{1}{\sqrt{1-t^2q^4p_{11}}}\,
(\de p_{11})
\,\frac{1}{\sqrt{1-t^2q^4p_{11}}}
\;.
\end{align*}
This means
\begin{subequations}
\begin{align}
A^* &=q\sqrt{1-t^2p_{11}}\;\de
\frac{1}{\sqrt{1-t^2p_{11}}}
\label{eq:1}
\;,\\[5pt]
B &=q^{-1}\sqrt{1-t^2q^4p_{11}}\;\deb
\frac{1}{\sqrt{1-t^2q^4p_{11}}}
\label{eq:3}
\;,\\[5pt]
\Phi &=-(\omega_t)_{21}^*=
-\frac{1}{(1-t^2p_{11})^{\frac{3}{2}}}\,\psi^\dag\,
\dd(\psi\psi^\dag)\varphi\,\frac{1}{\sqrt{1-t^2q^4p_{11}}} \;.
\label{eq:2}
\end{align}
\end{subequations}
The element $\Phi$ can be simplified.
An explicit computation gives:
\begin{align*}
-\psi^\dag\,\dd(\psi\psi^\dag)\varphi &=
\sin2\theta (\cos 2\theta\,z_1,z_2,z_3)\,
\dd\!\left(
\begin{pmatrix}
\cos 2\theta\,z_1^* \\ z_2^* \\ z_3^*
\end{pmatrix}
(z_2,z_3)
\right)
\binom{z_3}{-qz_2}
\\ &=
t\left(\sum\nolimits_{j=1}^3z_j\big\{(\dd p_{j2})z_3-q(\dd p_{j3})z_2\big\}
-t^2z_1\big\{(\dd p_{12})z_3-q(\dd p_{13})z_2\big\}\right) \;.
\end{align*}
The first term in the braces is just $q^{-\frac{1}{2}}\phi$. For the second term
one proves, using \eqref{eq:ddPij} and Lemma~\ref{prop:3.4}, that it's equal to $-t^2p_{11}q^{-\frac{1}{2}}\phi$.
Hence $-\psi^\dag\,\dd(\psi\psi^\dag)\varphi=t(1-t^2p_{11})q^{-\frac{1}{2}}\phi$ and
$$
\Phi=\frac{1}{\sqrt{1-t^2p_{11}}}\,
q^{-\frac{1}{2}}t\phi
\,\frac{1}{\sqrt{1-t^2q^4p_{11}}} \;.
$$
This concludes the proof of Proposition~\ref{eq:prop}.
\end{proof}


\section{Noncommutative families of instantons}

As observed in Sec.~\ref{sec:evb}, equation \eqref{eq:DeltaL} gives
a left coaction $\Delta_L$ of $\SU_q(3)$ on any module $\E(\sigma)$ of the type \eqref{eq:Esigma}.
In particular, the algebra of forms $\Omega^{\bullet,\bullet}$ is a graded
left $\SU_q(3)$-comodule $*$-algebra, with the coaction commuting with the exterior differential $\dd$, the latter
defined in \eqref{eq:2.5} via the left action $\mL{}$ of $\Uq$. In fact,
 $\Delta_L$ preserves also the decomposition $\Omega^{1,1}\simeq\Omega^{1,1}_v\oplus\Omega^{1,1}_s$, thus sending
ASD connections to ASD connections, with the Hodge star on the image of 
$\Delta_L$ defined by the trivial lift:
\begin{equation}\label{ext-star}
\hat{\star}_H:=\id_{\Oq}\otimes\Hs \;.
\end{equation}

\subsection{Adjoint coaction of $\SU_q(3)$ on connections}

If $\E$ is one of the modules $L_N$ or a direct sum of them,
we let $\Omega^{i,j}(\E)$ be the equivariant module of forms with
coefficients in $\E$, as explained in Sec.~\ref{sec:dfwcL}.
This is a module of the type \eqref{eq:Esigma}, and as such
is a $\SU_q(3)$-comodule for the coaction $\Delta_L$ in \eqref{eq:DeltaL},
that maps
 \begin{equation}\label{eq:DeltaE}
\Delta_L: \Omega^{i,j}(\E) \to \widehat{\Omega}^{i,j}(\E) :=\Oq\otimes\Omega^{i,j}(\E) \;.
\end{equation}
The space $\widehat{\Omega}^{i,j}(\E)$ is a right $\Omega^{\bullet,\bullet}(\CP^2_q)$-module for the wedge multiplication from the right on the factor $\Omega^{i,j}(\E)$ of the tensor product.
 Given a right module connection $\nabla$ on $\E$, the map $\id\otimes\nabla$ is a right module connection on $\widehat{\E}=\Oq\otimes\E$, and if $\nabla$ is (anti)self-dual, $\id\otimes\nabla$ is (anti)self-dual too for the Hodge star operator extended as in \eqref{ext-star}. 

The left coaction \eqref{eq:DeltaE} can be lifted to an adjoint coaction on connections.
We implement this by using the multiplicative unitary of $\SU_q(3)$ (in the sense of \cite{Wor96}). Consider then the ($\C$-linear) endomorphisms of $\widehat{\Omega}^{i,j}(\E)$ given by
$$
U:=(m\otimes\id)(\id\otimes\Delta_L) \;,\qquad
W:=(m\otimes\id)(\id\otimes S\otimes\id)(\id\otimes\Delta_L) \;,
$$
where $m$ is the multiplication map. Explicitly, for all $a\in\Oq$ and $\eta\in\Omega^{i,j}(\E)$:
$$
U(a\otimes \eta)=aS^{-1}(\eta_{(2)})\otimes \eta_{(1)} \;,\qquad
W(a\otimes \eta)=a\eta_{(2)}\otimes \eta_{(1)} \;,
$$
with $\Delta^{\mathrm{cop}}(\eta) = \eta_{(2)}\otimes\eta_{(1)}$ in Sweedler notation.
 Using $S^{-1}(\eta_{(2)})\eta_{(1)}=\eta_{(2)}S^{-1}(\eta_{(1)})=\epsilon(\eta)$,
one explicitly checks that $U$ and $W$ are one the inverse of the other.

\begin{lemma}
For any $a\in\Oq$, $\eta\in\Omega^{i,j}(\E)$ and $\omega\in\Omega^{i,j}(\CP^2_q)$ it holds that
\begin{equation}\label{eq:silly}
W(a\otimes\eta\omega)=W(a\otimes\eta)\Delta^{\mathrm{cop}}(\omega)
\end{equation}
where $\Delta^{\mathrm{cop}}(x)=x_{(2)}\otimes x_{(1)}$ is the opposite coproduct. 
\end{lemma}
\begin{proof}
>From $\Delta_L=(S^{-1}\otimes\id)\Delta^{\mathrm{cop}}$ it follows that
$W=(m\otimes\id)(\id\otimes\Delta^{\mathrm{cop}})$.
Since $\Delta^{\mathrm{cop}}$ is an algebra morphism,
we get \eqref{eq:silly}.
\end{proof}
Recall that $\Delta(\dd\omega)=(\dd\otimes\id)\Delta(\omega)$ for all $\omega\in\Omega^{\bullet,\bullet}$, expressing the (right) covariance of the calculus, due to the definition of $\dd$ using the right action of $\Uq$. Applying the flip we get
the covariance expressed for the opposite coproduct, that is
\begin{equation}\label{eq:sillyB}
\Delta^{\mathrm{cop}}(\dd\omega)=(\id\otimes\dd)\Delta^{\mathrm{cop}}(\omega) \;,
\end{equation}
a result that we shall use momentarily.

\begin{prop}
Let $\nabla$ be a right-module connection on $\E$. Then the operator
$$
\widehat{\nabla}:=U(\id\otimes\nabla)U^{-1} 
$$
is a right-module connection on $\widehat{\E}=\Oq\otimes\E$.
\end{prop}
\begin{proof}
Recall that $U^{-1}=W$, and explicitly $W(a\otimes\eta)=a\eta_{(2)}\otimes\eta_{(1)}$.
For any $a\in\Oq$, $\eta\in\Omega^{i,j}(\E)$ and $\omega\in\Omega^{i,j}(\CP^2_q)$,
one computes
\begin{align*}
W\widehat{\nabla}(a\otimes\eta\omega)
&=(\id\otimes\nabla)(a\eta_{(2)}\omega_{(2)}\otimes\eta_{(1)}\omega_{(1)}
=a\eta_{(2)}\omega_{(2)}\otimes\nabla(\eta_{(1)}\omega_{(1)})
\\[2pt]
&=a\eta_{(2)}\omega_{(2)}\otimes(\nabla\eta_{(1)})\omega_{(1)}
+(-1)^{\mathrm{dg}(\eta)}a\eta_{(2)}\omega_{(2)}\otimes\eta_{(1)}\dd\omega_{(1)} \;.
\end{align*}
On the other hand, from \eqref{eq:silly} and \eqref{eq:sillyB} we get
\begin{align*}
W\Big( \widehat{\nabla}(a\otimes\eta)\omega \Big)
&=\Big( W\widehat{\nabla}(a\otimes\eta) \Big) \Delta^{\mathrm{cop}}(\omega)
=a\eta_{(2)}\omega_{(2)}\otimes\nabla(\eta_{(1)})\omega_{(1)} \;,
\\
W\Big( (a\otimes\eta)\dd\omega \Big)
&=W(a\otimes\eta)\Delta^{\mathrm{cop}}(\dd\omega)
=(a\eta_{(2)}\otimes\eta_{(1)})(\id\otimes\dd)\Delta^{\mathrm{cop}}(\omega)
\\
&=a\eta_{(2)}\omega_{(2)}\otimes\eta_{(1)}\dd\omega_{(1)} \;.
\end{align*}
Thus
$$
W\widehat{\nabla}(\xi\omega)=
W\big( (\widehat{\nabla}\xi)\omega \big)+(-1)^{\mathrm{dg}(\xi)}W\big( \xi\dd\omega \big)
$$
for all $\xi=a\otimes\eta\in\widehat{\Omega}^{\bullet,\bullet}(\E)$ and $\omega\in\Omega^{\bullet,\bullet}$, and since $W$ is invertible, this
is equivalent to the graded Leibniz rule for $\widehat{\nabla}$.
\end{proof}

>From the expression for the square of the connection, 
$$
\widehat{\nabla}^2=U(\id\otimes\nabla^2)U ^{-1} \, ,
$$
we can easily work out the curvature two-form (-valued endomorphism) of $\widehat{\nabla}$.
If $F$ is the curvature two-form of $\nabla$, that is $\nabla^2\eta=F\eta$,
for any $a\otimes\eta\in\widehat{\Omega}^{\bullet,\bullet}(\E)$, we get:
\begin{align*}
\widehat{\nabla}^2(a\otimes\eta) &=U\big\{a\eta_{(2)}\otimes\nabla^2\eta_{(1)}\big\}
=U\big\{a\eta_{(2)}\otimes F\eta_{(1)}\big\}
\\
&=a\eta_{(2)}S^{-1}({\eta_{(1)}}_{(2)})S^{-1}(F_{(2)})\otimes F_{(1)}{\eta_{(1)}}_{(1)}
\\
&=a{\eta_{(2)}}_{(2)}S^{-1}({\eta_{(2)}}_{(1)})S^{-1}(F_{(2)})\otimes F_{(1)}\eta_{(1)}
\\
&=a (\epsilon(\eta_{(2)})) S^{-1}(F_{(2)})\otimes F_{(1)}\eta_{(1)}
=aS^{-1}(F_{(2)})\otimes F_{(1)}\eta \;.
\end{align*}
Where we used the coassociativity of $\Delta^{\mathrm{cop}}$: $\eta_{(2)} \otimes {\eta_{(1)}}_{(2)} \otimes {\eta_{(1)}}_{(1)} 
= {\eta_{(2)}}_{(2)}\otimes {\eta_{(2)}}_{(1)} \otimes \eta_{(1)}$, together with the properties 
${\eta_{(2)}}_{(2)}S^{-1}({\eta_{(2)}}_{(1)})=\epsilon(\eta_{(2)})$ and $\epsilon(\eta_{(2)}) \eta_{(1)} = \eta$.

Thus, the curvature two-form of $\widehat{\nabla}$ is given by
\begin{equation}\label{whf}
\widehat{F}:=\Delta_L(F) = S^{-1}(F_{(2)})\otimes F_{(1)}
\end{equation}
acting ``from the middle'' on $\widehat{\E}=\Oq\otimes\E$. 

As mentioned, the coaction $\Delta_L$ of $\SU_q(3)$ preserves the decomposition 
$\Omega^{1,1}=\Omega^{1,1}_v\oplus\Omega^{1,1}_s$. Hence, if $F$ is (anti)self-dual,
the curvature $\widehat{F}$ is  (anti)self-dual as well.

Next, suppose $\nabla \eta=\nabla_0\eta+\omega\eta$ where $\nabla_0$ is an invariant connection, that is, 
$$
(\id\otimes\nabla_0)\Delta_L(\eta)=\Delta_L(\nabla_0\eta) \;,
$$
and $\omega$ a one-form (-valued endomorphism). Then it follows immediately that 
$\widehat{\nabla}_0=\id\otimes\nabla_0$, while for any $a\otimes\eta\in\widehat{\E}$, 
proceeding as done before with $\widehat{\nabla}^2$: 
\begin{align*}
(\widehat{\nabla}-\widehat{\nabla}_0)(a\otimes\eta) 
& = U(a\eta_{(2)}\otimes\omega\eta_{(1)})
 = a\eta_{(2)}S^{-1}({\eta_{(1)}}_{(2)})S^{-1}(\omega_{(2)})\otimes\omega_{(1)}{\eta_{(1)}}_{(2)} \\
&=a{\eta_{(2)}}_{(2)}S^{-1}({\eta_{(2)}}_{(1)})S^{-1}(\omega_{(2)})\otimes \omega_{(1)}\eta_{(1)}
 = aS^{-1}(\omega_{(2)})\otimes\omega_{(1)}\eta \;.
\end{align*}
This can be read as the connection one-form of $\widehat{\nabla}$ be given by
\begin{equation}\label{who}
\widehat{\omega}:=\Delta_L(\omega)
\end{equation}
once again acting ``from the middle'' on $\widehat{\E}=\Oq\otimes\E$. 

\subsection{Coacting on instantons}
Let us know specialize the above discussion to the module $\E=L_1\oplus L_{-1}$, on which
there is a one-parameter family of right-module connections \mbox{$\widetilde{\nabla}_t=\widetilde{\nabla}_0+\omega_t$} given by equation \eqref{c1f},
with the connection one-form $\omega_t$ given in Proposition~\ref{eq:prop}.   
>From  $\widetilde{\nabla}_0=\nabla_1\oplus\nabla_{-1}$ and the expression of the monopole connections in \eqref{eq:conN}, it follows that $\widetilde{\nabla}_0$ is invariant. 
Thus
$$
U(\id\otimes\widetilde{\nabla}_t)U^{-1}=\id\otimes\widetilde{\nabla}_0+\widehat{\omega}_t \;,
$$
where $\widehat{\omega}_t=\Delta_L(\omega_t)$ acts ``from the middle'' on $\Oq\otimes \left( L_1\oplus L_{-1} \right)$.

Although $\widehat{\omega}_t\in\Oq\otimes\Omega^1(L_1\oplus L_{-1})$, elements in the left leg of the connection one-form generate an algebra smaller than $\Oq$, that is what we want now to determine.

>From \eqref{eq:prop1}-\eqref{eq:prop3} we see that the matrix entries of $\omega_t$
are linear combinations of $\de x$, $\deb x$, $\phi$ and $\phi^*$, with $x=p_{11}$
and $\phi$ as in \eqref{eq:phip33}. From the mentioned equation we get
\begin{gather*}
\Delta_L(x)=\sum\nolimits_{j,k}q^{2(k-1)}(u^1_k)^*u^1_j\otimes p_{jk} \;,\qquad
\Delta_L(\de x)=\sum\nolimits_{j,k}q^{2(k-1)}(u^1_k)^*u^1_j\otimes \de p_{jk} \;, \\
\Delta_L(\phi^*)=\ii \sum\nolimits_jq^{2(j-1)}(u^1_j)^*\otimes\binom{u^2_j}{u^1_j} \;.
\end{gather*}
where we used $S^{-1}((u^i_j)^*)=u^j_i$ and $S^{-1}(u^i_j)=q^{2(i-j)}(u^j_i)^*$.
Since by Rem.~\ref{rem:otherS} the \mbox{$*$-algebra} generated by the elements $\{u^1_i\}_{i=1,2,3}$ and their
conjugated is isomorphic to $\Sq$, for any $t\neq 0$ we get a non-commutative space of ASD connections isomorphic to $\mathrm{S}^5_q$.

On the other hand, let $g$ be the generator of $\A(\mathrm{U}(1))$. The group $\mathrm{U}(1)$
is a quantum subgroup of $\SU_q(3)$; the surjective Hopf $*$-algebra morphism
$\Oq\to\A(\mathrm{U}(1))$ is given by $u^i_j\mapsto\delta_{i,j}g^{2-j}$ (dual to the diagonal inclusion $e^{i\theta}\mapsto\mathrm{diag}(e^{i\theta},1,e^{-i\theta})$ for $q=1$).

The coaction of the quantum subgroup $\mathrm{U}(1)$ leaves $x=p_{11}$ invariant
and maps $\phi$ to $g\otimes\phi$ and $\phi^*$ to $g^*\otimes\phi^*$.
Thus
$$
\omega_t\mapsto\maa{1\otimes 1 & 0 \\ 0 & g^*\otimes 1}(1\otimes\omega_t)
\maa{1\otimes 1 & 0 \\ 0 & g\otimes 1}
$$
corresponds to a (global) gauge transformation and can be neglected.

Summing up, modulo gauge transformations, for any $t\neq 0$ we have a noncommutative family
of ASD connections parametrized by $\mathrm{S}^5_q/\mathrm{U}(1)\simeq\CP^2_q$.


\section{Classical results from a noncommutative view-point}\label{sec:AppA}

In this section $q=1$, that is we deal with the classical projective space $\CP^n$. 
We restate some of the geometrical properties of $\CP^n$ from a noncommutative view-point,
so as to readily generalized them to the noncommutative deformations.

\subsection{The K{\"a}hler form of classical $\CP^2$} 
Let $[z_1,\ldots,z_{n+1}]$ be homogeneous coordinates on $\CP^n$.
 On the chart $U_\alpha:=\{z_\alpha\neq 0\}\simeq\C^n$ there are complex
coordinates
\begin{equation}\label{eq:locxi}
Z_\beta=z_\beta/z_\alpha \;,\quad\forall\;\beta\neq\alpha\;,
\end{equation}
and the transition functions are holomorphic on the intersections.

As a real manifold $\CP^n$ is diffeomorphic to the set $\mathfrak{M}$ of those matrices 
$p\in M_{n+1}(\C)$ such that $p=p^*=p^2$ and $\sum_ip_{ii}=1$. The map 
$\CP^n\to\mathfrak{M}$ is given by
\begin{equation}\label{eq:pij}
[z_1,\ldots,z_{n+1}]\mapsto p_{ij}=||z||^{-2}\bar z_iz_j \;.
\end{equation}
where $||z||^2=\sum_k|z_k|^2$.
The inverse map sends $p$ to the point $[z_1,\ldots,z_{n+1}]$
defined as the equivalence class of any non-zero row of $p$
(since $p$ is a rank $1$ projection, $p\neq 0$ and it has always at least one non-zero row).
One could restrict the homogeneous coordinates to $(z_1,\ldots,z_{n+1})\in S^{2n+1}$ 
and components of the projection in \eqref{eq:pij} would just be $p_{ij}=\bar z_iz_j$.

On the chart $U_\alpha$, the K{\"a}hler form associated to the Fubini-Study metric is:
 \begin{align*}
\kahler &=\frac{ \ii }{2}\left\{\frac{1}{1+||Z||^2}\sum \de Z_\beta\wedge\deb\bar Z_\beta
-\frac{1}{(1+|| Z||^2)^2}\sum \bar Z_\beta\de Z_\beta\wedge  Z_\gamma\deb\bar Z_\gamma
\right\} \\
&=\frac{ \ii}{2}\left\{\frac{1}{1+|| Z||^2}\sum \de Z_\beta\wedge\deb\bar Z_\beta
-\frac{1}{(1+|| Z||^2)^2}\de|| Z||^2\wedge \deb|| Z||^2\right\}
 \;
\end{align*}
(cf.~Example 4.5, page 189 of \cite{Wel80}). 
On the other hand, using \eqref{eq:pij}:
\begin{align*}
\de p_{ij} &=\frac{1}{1+|| Z||^2}\bar Z_i\de Z_j-
\frac{1}{(1+|| Z||^2)^2}\bar Z_i Z_j\de|| Z||^2 \;,\\
\deb p_{ji} &=\frac{1}{1+|| Z||^2} Z_i\deb\bar Z_j-
\frac{1}{(1+|| Z||^2)^2}\bar Z_j Z_i\deb|| Z||^2 \;,
\end{align*}
where we set $ Z_i=1$ if $i=\alpha$.
One easily checks that:
$$
\kahler=\frac{ \ii }{2}\sum\nolimits_{ij}\de p_{ij} \wedge\deb p_{ji} \;.
$$
In our notations for $n=2$, using \eqref{eq:useful} with $q=1$,
we get
\begin{equation}\label{eq:kahlerq}
\kahler=\frac{ \ii }{2}\,\omega_{q=1} \;.
\end{equation}


\subsection{Deconstructing instantons on classical $\CP^2$} 
Let us work on the chart $U_1$ with coordinates $Z_2=z_2/z_1$ and $Z_3=z_3/z_1$.
In the notations of \cite{Hab92}, the homogeneous coordinates there are
$[T_0,T_1,T_2]=[z_1,z_3,z_2]$, and the local coordinates $z_1,z_2$ 
are our $Z_3,Z_2$.

For any $0\leq t<1$, on the direct sum of the tautological bundle with its dual there is an ASD connection with connection one-form given
by (cf.~\cite{Gro90}):
$$
\omega_t=
\frac{1}{1+||Z||^2-t^2}\begin{pmatrix}
\frac{1}{2}(\deb||Z||^2-\de||Z||^2)
&
t(Z_3\dd Z_2-Z_2\dd Z_3) \\[4pt]
* &
*
\end{pmatrix} \;.
$$
where $||Z||^2:=Z_2\bar Z_2+Z_3\bar Z_3$ and the second row is obtained from the first one being $\omega_t$ traceless and anti-hermitian.

Using the parametrization $t=\sin 2\theta$, with $0\leq\theta<\pi/4$, the matrix of functions   
$$
\Psi:=
\frac{1}{\sqrt{(\cos 2\theta)^2+||Z||^2}}
\left(\begin{array}{c|c}
\cos 2\theta & 0 \\
\cos\theta\,\bar Z_2 & -\sin\theta\,Z_3 \\
\cos\theta\,\bar Z_3 & \sin\theta\,Z_2 \\
\hline
0 & \cos 2\theta \\
\sin\theta\,\bar Z_3 & \cos\theta\,Z_2 \\
\sin\theta\,\bar Z_2 & -\cos\theta\,Z_3
\end{array}\right) \;
$$
is normalized, that is $\Psi^\dag\Psi=1_2$. One checks that $\omega_t=\Psi^\dag\dd\Psi$; hence, $\omega_t$ is the connection one-form of the Grassmannian connection of the projection $P:=\Psi\Psi^\dag$.

Notice that $1=\sum_iz_iz_i^*=z_1z_1^*(1+||Z||^2)$,
so $1+||Z||^2=p_{11}^{-1}$ and
$$
\Psi=
\frac{1}{\sqrt{1-t^2p_{11}}}\begin{pmatrix}[rr]
\cos 2\theta\,z_1^* & 0 \hspace{6mm} \\
\cos\theta\,z_2^* & -\sin\theta\,z_3 \\
\cos\theta\,z_3^* & \sin\theta\,z_2 \\
0 \hspace{6mm} & \cos 2\theta\, z_1 \\
\sin\theta\,z_3^* & \cos\theta\,z_2 \\
\sin\theta\,z_2^* & -\cos\theta\,z_3
\end{pmatrix}
\begin{pmatrix}
z_1/|z_1| & 0 \\
0 & z_1^*/|z_1|
\end{pmatrix} \;.
$$
Since $\Psi$ has no limit for $z_1\to 0$, it cannot be extended
to $\mathrm{S}^5$. On the other hand, with a slight modification we can get
a matrix of functions on $\mathrm{S}^5$:
\begin{equation}\label{eq:decon}
\widetilde{\Psi}:=
\frac{1}{\sqrt{1-t^2p_{11}}}\begin{pmatrix}[rr]
\cos 2\theta\,z_1^* & 0 \hspace{6mm} \\
\cos\theta\,z_2^* & -\sin\theta\,z_3 \\
\cos\theta\,z_3^* & \sin\theta\,z_2 \\
0 \hspace{6mm} & \cos 2\theta\, z_1 \\
\sin\theta\,z_3^* & \cos\theta\,z_2 \\
\sin\theta\,z_2^* & -\cos\theta\,z_3
\end{pmatrix} \;.
\end{equation}
Since $\Psi\Psi^\dag=\widetilde{\Psi}\widetilde{\Psi}^\dag$, the matrix
$\widetilde{\Psi}$ (or rather, its analogue for $q\neq 1$) is our
starting point for the construction of instantons in the noncommutative
case of $\CP^2_q$. The Grassmannian connection associated to $P:=\widetilde{\Psi}\widetilde{\Psi}^\dag$, transported
on the equivariant module whose generators are the components of $\widetilde{\Psi}^\dag$,
is an instanton for $q=1$.


\appendix

\section{The proof of Theorem~\ref{thm}}\label{app:B}
We collect here the details of the proof of Theorem~\ref{thm} that is to say that the curvature
$$
F=\Psi^\dag(\dd P\wprod \dd P)\Psi \:, 
$$
with $\Psi$ the $2\times 6$ matrix in \eqref{eq:Psitilde} (we will omit the tilde) and $P=\Psi\Psi^\dag$, 
is ASD.

For this we shall need to take derivatives of elements which go beyond the ones in $\Aq$.
Now, the operators $\de$ and $\deb$ in \eqref{eq:2.5} can be extended in the obvious way
to maps \mbox{$\Oq\to\Oq^2$}, that we denote by the same symbols. However, $\de$ and $\deb$
are not derivations on $\Oq$ (if $q\neq 1$), but rather twisted ones coming from \eqref{glr},
nor their square is zero. For future use, we compute
\begin{equation}\label{eq:dezdebz}
\de z_j= \ii q^{-\frac{3}{2}}\binom{u^2_j}{u^1_j} \;,\qquad
\deb z_j^*= \ii q^{-\frac{3}{2}}\binom{q^{-\frac{1}{2}}(u^1_j)^*}{-q^{\frac{1}{2}}(u^2_j)^*} \;.
\end{equation}
Also, here and in the following we set $x:=p_{11}$ and call $f(x)$ the element
$$
f(x):=\frac{1}{\sqrt{1-t^2x}} \;.
$$

\begin{lemma}
The curvature can be written as
\begin{equation}\label{lemma:B1}
F=\big\{(\Psi^\dag\dd\Psi)Q\big\}^2+Q(\dd\Psi^\dag)\wprod (\dd\Psi)Q \;.
\end{equation}
with
\begin{equation}\label{eq:Q}
Q=\maa{ q^{\frac{1}{2}} & 0\; \\ 0\; & \;q^{-\frac{1}{2}} } \;.
\end{equation}
\end{lemma}

\begin{proof}
Using
$$
\Delta(E_i)=E_i\otimes K_i+K_i^{-1}\otimes E_i\;, \quad 
\Delta(F_i)=F_i\otimes K_i+K_i^{-1}\otimes F_i\;,
$$
and $z_i\za K_1=z_i$, $z_i\za K_2=q^{\frac{1}{2}}z_i$, one finds
$$
\dd P=(\dd\Psi)(\Psi^\dag\za K_2)+(\Psi\za K_2^{-1})(\dd\Psi^\dag) \;.
$$
Since
$$
\Psi\za K_2^{-1}=\Psi\maa{ q^{\frac{1}{2}} & 0\; \\ 0\; & \;q^{-\frac{1}{2}} } \;,\qquad
\Psi^\dag\za K_2=\maa{ q^{\frac{1}{2}} & 0\; \\ 0\; & \;q^{-\frac{1}{2}} }\Psi^\dag \;, 
$$
we get
\begin{equation}\label{eq:again}
\dd P=(\dd\Psi)Q\Psi^\dag+\Psi Q(\dd\Psi^\dag) \;,
\end{equation}
with $Q$ as in \eqref{eq:Q}.
Since $\Psi=P\Psi$, then
$$
(\dd P)\Psi=(\dd\Psi)Q+P\Psi Q(\dd\Psi^\dag)\Psi \;.
$$
The second term gives no contribution to the curvature $F$, due to
$$
\Psi^\dag(\dd P)\wprod P\Psi Q(\dd\Psi^\dag)\Psi=
\Psi^\dag(\dd P)P\wprod \Psi Q(\dd\Psi^\dag)\Psi \;,
$$
but $(\dd P)P=(1-P)\dd P$, and then $\Psi^\dag(\dd P)P=0$.
Hence
$$
F=\Psi^\dag(\dd P)\wprod (\dd\Psi)Q \;.
$$
Using again \eqref{eq:again} we get the thesis.
\end{proof}

\noindent
We now start computing the many pieces in \eqref{lemma:B1}.

\begin{lemma}
Since $z_i\za F_2=z_i^*\za E_2=0$, we have
$$
\dd\psi=\deb\psi \;,\qquad \dd\varphi=\de\varphi \;,
$$
and a straightforward computation gives, 
\begin{equation}\label{eq:xxx}
\dd\Psi=(\deb\psi , 0)f(x) +(0,\de\varphi)f(q^4x) +q^{\frac{1}{2}}(\psi , 0)\dd f(x) +q^{-\frac{1}{2}}(0,\varphi)\dd f(q^4x) \;. 
\end{equation}
\end{lemma}

\begin{lemma}\label{lemma:B3}
\begin{align*}
\psi^\dag\deb\psi &=q^{\frac{1}{2}}\deb f(x)^{-2}=-q^{\frac{1}{2}}t^2\deb x  \;, &
\varphi^\dag\de\varphi &=q^{-\frac{1}{2}}\de f(q^4x)^{-2}=-q^{-\frac{1}{2}}t^2q^4\de x \;,
\\
(\de\psi^\dag)\psi &=q^{\frac{1}{2}}\de f(x)^{-2}=-q^{\frac{1}{2}}t^2\de x \;, &
(\deb\varphi^\dag)\varphi &=q^{-\frac{1}{2}}\deb f(q^4x)^{-2}=-q^{-\frac{1}{2}}t^2q^4\deb x  \;.
\end{align*}
\end{lemma}
\begin{proof}
Since $\deb\psi^\dag=0$, one has $\deb(\psi^\dag\psi)=q^{-\frac{1}{2}}\psi^\dag\deb\psi$. From $\psi^\dag\psi=f(x)^{-2}=1-t^2x$
the first equation follows.
Similarly, from $\de(\varphi^\dag\varphi)=q^{\frac{1}{2}}\varphi^\dag\de\varphi$ and $\varphi^\dag\varphi=f(q^4x)^{-2}=1-t^2q^4x$
one gets the second one. The remaining equations are proved similarly.
\end{proof}

\begin{lemma}\label{cor:B4}
$$
(\Psi^\dag\dd\Psi)Q=
\begin{pmatrix}
qf(x)^{-1}\de f(x)+qf(x)\deb f(x)^{-1}  &
q^{-\frac{1}{2}}f(x)(\psi^\dag\de\varphi)f(q^4x) \\[5pt]
q^{\frac{1}{2}}f(q^4x)(\varphi^\dag\deb\psi)f(x) &
q^{-1}f(q^4x)^{-1}\deb f(q^4x)+q^{-1}f(q^4x)\de f(q^4x)^{-1}
\end{pmatrix}
$$
\end{lemma}
\begin{proof}
Let $M:=(\Psi^\dag\dd\Psi)Q$. From \eqref{eq:xxx} and 
$$
\Psi^\dag=\binom{f(x)\psi^\dag}{f(q^4x)\varphi^\dag} \:,
$$
one easily computes the components $M_{12}$ and $M_{21}$, as well as:
\begin{align*}
q^{-1}M_{11} &=f(x)^{-1}\dd f(x)+f(x)q^{-\frac{1}{2}}(\psi^\dag\deb\psi)f(x) \;, \\
qM_{22} &=f(q^4x)^{-1}\dd f(q^4x)+f(q^4x)(q^{\frac{1}{2}}\varphi^\dag\de\varphi)f(q^4x) \;,
\end{align*}
where we used the normalizations $\psi^\dag\psi=f(x)^{-2}$ and $\varphi^\dag\varphi=f(q^4x)^{-2}$. Using the first
equation in Lemma~\ref{lemma:B3} and some algebraic manipulation with the Leibniz rule, we get
\begin{align*}
q^{-1}M_{11} &=f(x)^{-1}\dd f(x)+f(x)(\deb f(x)^{-2})f(x) \\
&=f(x)^{-1}\de f(x)+f(x)^{-1}\deb f(x)+(\deb f(x)^{-1})f(x)+f(x)\deb f(x)^{-1} \\
&=f(x)^{-1}\de f(x)+f(x)\deb f(x)^{-1} \;.
\end{align*}
Note that in the second line we used that
$$
f(x)^{-1}\deb f(x)+(\deb f(x)^{-1})f(x)=\deb\big\{f(x)^{-1}f(x)\big\}=\deb 1=0 \;.
$$
In the same way, using the second equation in Lemma~\ref{lemma:B3} and some algebraic manipulation with the Leibniz rule, we get
\begin{align*}
qM_{22} &=f(q^4x)^{-1}\dd f(q^4x)+f(q^4x)(\de f(q^4x)^{-2})f(q^4x) \\
&=f(q^4x)^{-1}\deb f(q^4x)+f(q^4x)^{-1}\de f(q^4x)+(\de f(q^4x)^{-1})f(q^4x)+f(q^4x)\de f(q^4x)^{-1} \\
&=f(q^4x)^{-1}\deb f(q^4x)+f(q^4x)\de f(q^4x)^{-1} \;.
\end{align*}
This concludes the proof.
\end{proof}

Let us denote by $\omega^+$ the self-dual part of a $(1,1)$-form $\omega$. Recall that it belongs to a subspace of $\Oq\otimes\C^3$:
we will write its components in a column.

\begin{lemma}\label{lemma:1}
With $\psi$ and $\varphi$ the vector-valued functions in \eqref{pf} it holds that
\begin{align}\label{eq:yyy}
\psi^\dag\de\varphi & = - q (\de\psi^\dag)\varphi = t \phi  \:, \nonumber \\
\varphi^\dag\deb\psi & = - q^{-1} (\deb\varphi^\dag)\psi = t q^{-1} \phi^* \;, 
\end{align}
where $\phi$ is the particular element of $L_2\otimes_{\Aq}\Omega^1$ given in \eqref{eq:phi}.
\end{lemma}
\begin{proof}
A simple computation gives, 
$$
\psi^\dag\de\varphi = t(qz_3\de z_2-z_2\de z_3) \;, \qquad \varphi^\dag\deb\psi =t(qz_2^*\deb z_3^*-z_3^*\deb z_2^*) \;.
$$
Then, from \eqref{eq:dezdebz} and \eqref{eq:star} we get, 
\begin{align*}
qz_3\de z_2-z_2\de z_3 &= \ii q^{-\frac{3}{2}}\binom{qu^3_3u^2_2-u^3_2u^2_3}{qu^3_3u^1_2-u^3_2u^1_3}
     = \ii q^{-\frac{3}{2}}\binom{qu^2_2u^3_3-q^2u^2_3u^3_2}{qu^1_2u^3_3-q^2u^1_3u^3_2}= \ii \binom{q^{-\frac{1}{2}}(u^1_1)^*}{-q^{\frac{1}{2}}(u^2_1)^*} \\
qz_2^*\deb z_3^*-z_3^*\deb z_2^* & = \ii q^{-\frac{3}{2}}\binom{q^{-\frac{1}{2}}(qu^1_3u^3_2-u^1_2u^3_3)^*}{-q^{\frac{1}{2}}(qu^2_3u^3_2-u^2_2u^3_3)^*}
     = q^{-1} \binom{u^2_1}{u^1_1}
\end{align*}
and the left hand sided in \eqref{eq:yyy} follow from a comparison with \eqref{eq:phip33}.
\end{proof}

\begin{lemma}
\begin{subequations}
\begin{align}
(\phi\wprod\phi^*)^+ &= q^{\frac{1}{2}s}c_1
\begin{bmatrix}
q^{-\frac{1}{2}}(u^1_1)^*u^2_1 \\[3pt]
[2]^{-\frac{1}{2}}\big(q^{-1}(u^1_1)^*u^1_1-q(u^2_1)^*u^2_1 \big) \\[3pt]
-q^{\frac{1}{2}}(u^2_1)^*u^1_1
\end{bmatrix} \label{eq:consequence}
\\[5pt]
(\phi^*\wprod\phi)^+ &= -c_1\begin{bmatrix}
q^{-\frac{1}{2}}u^2_1(u^1_1)^* \\[3pt]
[2]^{-\frac{1}{2}}\big(u^1_1(u^1_1)^*-u^2_1(u^2_1)^* \big) \\[3pt]
-q^{\frac{1}{2}}u^1_1(u^2_1)^*
\end{bmatrix} \label{eq:consequenceB}
\end{align}
\end{subequations}
\end{lemma}
\begin{proof}
It follows from Lemma~\ref{lemma:1} after a simple computation.
\end{proof}

\begin{lemma}
\begin{subequations}
\begin{align}
(\de x\wprod\deb x)^+ &=q^{\frac{1}{2}s}c_1q^{-2}x
\begin{bmatrix}
q^{-\frac{1}{2}}u^2_1(u^1_1)^* \\[3pt]
[2]^{-\frac{1}{2}}\big( u^1_1(u^1_1)^*-u^2_1(u^2_1)^* \big) \\[3pt]
-q^{\frac{1}{2}}u^1_1(u^2_1)^*
\end{bmatrix} \label{eq:dexdebx}
\\[5pt]
(\deb x\wprod\de x)^+ &=-q^{-4}c_1x
\begin{bmatrix}
q^{-\frac{1}{2}}(u^1_1)^*u^2_1 \\[3pt]
[2]^{-\frac{1}{2}}\big(q^{-1}(u^1_1)^*u^1_1-q(u^2_1)^*u^2_1\big) \\[3pt]
-q^{\frac{1}{2}}(u^2_1)^*u^1_1
\end{bmatrix} \label{eq:dexdebxB}
\end{align}
\end{subequations}
\end{lemma}
\begin{proof}
Note that
\begin{align*}
\de x\wprod\deb x &=z_1^*\de z_1\wprod z_1\deb z_1^*=qx\de z_1\wprod\deb z_1^* \\[3pt]
\deb x\wprod\de x &=z_1\deb z_1^*\wprod z_1^*\de z_1=q^{-1}x\deb z_1^*\wprod\de z_1 \;.
\end{align*}
The computation of $\de z_1\wprod\deb z_1^*$ and $\deb z_1^*\wprod\de z_1$ is straightforward.
\end{proof}

\begin{lemma}
\begin{subequations}
\begin{align}
(\de\psi^\dag\wprod\deb\psi)^+ &=
-q^{-3+\frac{1}{2}s}c_1t^2
\begin{bmatrix}
q^{-\frac{1}{2}}u^2_1(u^1_1)^* \\[3pt]
[2]^{-\frac{1}{2}}\big(u^1_1(u^1_1)^*-u^2_1(u^2_1)^*\big) \\[3pt]
-q^{\frac{1}{2}}u^1_1(u^2_1)^*
\end{bmatrix} \label{eq:dedebp}
\\
 (\deb\varphi^\dag\wprod\de\varphi)^+ &=qc_1t^2
\begin{bmatrix}
q^{-\frac{1}{2}}(u^1_1)^*u^2_1 \\[3pt]
[2]^{-\frac{1}{2}}\big( q^{-1}(u^1_1)^*u^1_1-q(u^2_1)^*u^2_1 \big) \\[3pt]
-q^{\frac{1}{2}}(u^2_1)^*u^1_1
\end{bmatrix}
 \label{eq:debdep}
 \end{align}
\end{subequations}
\end{lemma}
\begin{proof}
One has
\begin{align*}
\de\psi^\dag\wprod\deb\psi &=-t^2\de z_1\wprod\deb z_1^*+\sum\nolimits_{j=1}^3\de z_j\wprod\deb z_j^*
\\
&=q^{-3}t^2\binom{u^2_1}{u^1_1}\wprod\binom{q^{-\frac{1}{2}}(u^1_1)^*}{-q^{\frac{1}{2}}(u^2_1)^*}
-q^{-3}\sum\nolimits_{j=1}^3\binom{u^2_j}{u^1_j}\wprod\binom{q^{-\frac{1}{2}}(u^1_j)^*}{-q^{\frac{1}{2}}(u^2_j)^*}
\\[10pt]
\deb\varphi^\dag\wprod\de\varphi &=-t^2q^4\deb z_1^*\wprod\de z_1+\sum\nolimits_{j=1}^3q^{6-2j}\deb z_j^*\wprod\de z_j
\\
&=q^{-3}t^2q^4\binom{q^{-\frac{1}{2}}(u^1_1)^*}{-q^{\frac{1}{2}}(u^2_1)^*}\wprod\binom{u^2_1}{u^1_1}
-q^{-3}\sum\nolimits_{j=1}^3q^{6-2j}\binom{q^{-\frac{1}{2}}(u^1_j)^*}{-q^{\frac{1}{2}}(u^2_j)^*}\wprod\binom{u^2_j}{u^1_j}
\end{align*}
Using Proposition~\ref{pr:const}:
\begin{multline*}
q^{3-\frac{1}{2}s}c_1^{-1}(\de\psi^\dag\wprod\deb\psi)^+
=-t^2
\begin{bmatrix}
q^{-\frac{1}{2}}u^2_1(u^1_1)^* \\[3pt]
[2]^{-\frac{1}{2}}\big(u^1_1(u^1_1)^*-u^2_1(u^2_1)^*\big) \\[3pt]
-q^{\frac{1}{2}}u^1_1(u^2_1)^*
\end{bmatrix}
\\
+\sum\nolimits_{j=1}^3\begin{bmatrix}
q^{-\frac{1}{2}}u^2_j(u^1_j)^* \\[3pt]
[2]^{-\frac{1}{2}}\big(u^1_j(u^1_j)^*-u^2_j(u^2_j)^*\big) \\[3pt]
-q^{\frac{1}{2}}u^1_j(u^2_j)^*
\end{bmatrix}
\end{multline*}
It follows from Lemma~\ref{lemma:ortho} that the second term is zero, and this proves \eqref{eq:dedebp}. Similarly
\begin{multline*}
q^3c_1^{-1}(\deb\varphi^\dag\wprod\de\varphi)^+
=t^2q^4
\begin{bmatrix}
q^{-\frac{1}{2}}(u^1_1)^*u^2_1 \\[3pt]
[2]^{-\frac{1}{2}}\big( q^{-1}(u^1_1)^*u^1_1-q(u^2_1)^*u^2_1 \big) \\[3pt]
-q^{\frac{1}{2}}(u^2_1)^*u^1_1
\end{bmatrix}
\\
-\sum\nolimits_{j=1}^3q^{6-2j}\begin{bmatrix}
q^{-\frac{1}{2}}(u^1_j)^*u^2_j \\[3pt]
[2]^{-\frac{1}{2}}\big( q^{-1}(u^1_j)^*u^1_j-q(u^2_j)^*u^2_j \big) \\[3pt]
-q^{\frac{1}{2}}(u^2_j)^*u^1_j
\end{bmatrix}
\end{multline*}
It follows from Lemma 2.1 that the second term is zero, and this proves \eqref{eq:debdep}.
\end{proof}

As mentioned $(2,0)$ forms and $(0,2)$ forms are ASD, thus we only need to consider the $(1,1)$ component of the curvature $F$. We denote by $F_{ij}$ the matrix elements of the $\Omega^{1,1}$-component of $F$. Since $F_{21}=F_{12}^*$, we only need to compute three matrix elements.

\begin{lemma}\label{lemma:B7}
\begin{align*}
z_1 \phi &=q \phi z_1 \:, & z_1^* \phi &=q \phi z_1^* \:, &
z_1 \phi^* &=q^{-1} \phi^* z_1 \:, & z_1^* \phi^* &=q^{-1} \phi^* z_1^* \:.
\end{align*}
\end{lemma}
\begin{proof}
The proof is an easy computation, similar to the proof of \eqref{eq:lemma36A}.
\end{proof}

\begin{lemma}\label{lem:12}
The off-diagonal terms of $F$ vanish:
\begin{align}
F_{12}=0 \; .
\end{align}
\end{lemma}
\begin{proof}
>From Lemma~\ref{lemma:B1} and Lemma~\ref{cor:B4}:
\begin{align*}
q^{\frac{1}{2}}F_{12} &= - t q f(x)\deb f(x)\wprod\phi+itq^{-1}f(x)\phi\wprod \deb f(q^4x) \\[3pt]
&\qquad\qquad  - t q^{-1} f(x)\phi\wprod\deb f(q^4x)+t q \deb f(x)\wprod \phi f(q^4x) \\[3pt]
&=t q \big[\deb f(x)\wprod \phi\big]f(q^4x)-t q f(x)\big[\deb f(x)\wprod\phi\big] \;.
\end{align*}
Since $\phi x=q^{-2}x\phi$ and $(\deb x)x=q^{-2}x(\deb x)$ (cf.~Lemma~\ref{lemma:B7} and \eqref{eq:lemma36}) last term is zero.
\end{proof}

\begin{lemma}\label{lemma:B12}
\begin{align*}
F_{11} &=qf(x)^2\de\psi^\dag\wprod\deb\psi+t^2q^{-1}f(x)^2f(q^2x)^2\phi\wprod\phi^* \\[3pt]
&-q^2t^2\big\{f(x)+f(q^2x)\big\}\dot{f}(qx)\de x\wprod\deb x \:.
\end{align*}
\end{lemma}

\begin{proof}
>From Lemma~\ref{lemma:B1}, Lemma~\ref{cor:B4} and equation \eqref{eq:yyy}:
\begin{align*}
q^{-2}F_{11} &=q^{-1}f(x)^2\de\psi^\dag\wprod\deb\psi+t^2q^{-3}f(x)^2f(q^2x)^2\phi\wprod\phi^* \\
&+f^{-1}\de f\wprod f\deb f^{-1}+\de f\wprod (\deb f^{-2})f+f\de f^{-2}\wprod \deb f+\de f\wprod f^{-2}\deb f
 \\
&+f\deb f^{-1}\wprod f^{-1}\de f+\deb f\wprod f^{-2}\de f
\end{align*}
where $f=f(x)$. But
$$
f\deb f^{-1}\wprod f^{-1}\de f+\deb f\wprod f^{-2}\de f=\deb(ff^{-1})\wprod f^{-1}\de f=0 \;.
$$
Also
$$
\de f\wprod (\deb f^{-2})f+\de f\wprod f^{-2}\deb f=\de f\wprod\deb f^{-1}
$$
Therefore
\begin{align*}
q^{-2}F_{11} &=q^{-1}f(x)^2\de\psi^\dag\wprod\deb\psi + t^2q^{-3}f(x)^2f(q^2x)^2\phi\wprod\phi^* \\
&+f^{-1}\de f\wprod f\deb f^{-1} +f\de f^{-2}\wprod \deb f + \de f\wprod\deb f^{-1}
\end{align*}
But $f^{-1}(\de f)f+\de f=f^{-1}\de f^2=-f(\de f^{-2})f^2$. 
Hence
\begin{align*}
q^{-2}F_{11} &=q^{-1}f(x)^2\de\psi^\dag\wprod\deb\psi+t^2q^{-3}f(x)^2f(q^2x)^2\phi\wprod\phi^* \\
&+f\de f^{-2}\wprod (-f^2\deb f^{-1} +\deb f)
\end{align*}
Since $\de f^{-2}=-t^2\de x$, from
$$
f\de f^{-2}\wprod (-f^2\deb f^{-1} +\deb f)=-t^2\big\{f(x)+f(q^2x)\big\}\dot{f}(qx)\de x\wprod\deb x
$$
the thesis follows. Note that we used Cor.~\ref{cor:qder} and the commutation rules in \eqref{eq:lemma36}.
\end{proof}

\begin{lemma}\label{lemma:eta}
Let $\eta$ be the following self-dual $(1,1)$-form:
$$
\eta:=-q^{\frac{1}{2}s}c_1
\begin{bmatrix}
q^{-\frac{1}{2}}u^2_1(u^1_1)^* \\[3pt]
[2]^{-\frac{1}{2}}\big(u^1_1(u^1_1)^*-u^2_1(u^2_1)^*\big) \\[3pt]
-q^{\frac{1}{2}}u^1_1(u^2_1)^*
\end{bmatrix}
$$
Then
$$
(\de\psi^\dag\wprod\deb\psi)^+ =q^{-3}t^2\eta \;,\quad\quad
(\de x\wprod\deb x)^+ =-q^{-2}x\eta \;,\quad\quad
(\phi\wprod\phi^*)^+= - q^{-1}\eta
\;.
$$
\end{lemma}

\begin{proof}
This follows by comparing \eqref{eq:consequence}, \eqref{eq:dexdebx} and \eqref{eq:dedebp}.
The commutation relations between $u^i_1$ and $(u^j_1)^*$ are obtained from the observation that
$z^i:=(u^i_1)^*$ satisfy the same commutation rules of $\Sq$ (cf.~Remark~\ref{rem:otherS}).
\end{proof}

\begin{cor}\label{cor:6} 
The self-dual part of $F_{11}$ vanishes:
$$
F_{11}^+=0
$$
\end{cor}
\begin{proof}
>From Lemma~\ref{lemma:B12} and Lemma~\ref{lemma:eta}, we deduce that $F_{11}^+=t^2a\eta$, where $a\in\Aq$ is the following element:
$$
a=q^{-2}f(x)^2-q^{-2}f(x)^2f(q^2x)^2+x\big\{f(x)+f(q^2x)\big\}\dot{f}(qx) \;.
$$
But
$$
q^{-2}f(x)^2-q^{-2}f(x)^2f(q^2x)^2=\frac{-t^2x}{(1-t^2x)(1-t^2q^2x)}
$$
and
\begin{multline*}
x\big\{f(x)+f(q^2x)\big\}\dot{f}(qx)=\big\{f(x)+f(q^2x)\big\}\frac{f(q^2x)-f(x)}{q^2-1}=\frac{f(q^2x)^2-f(x)^2}{q^2-1}
\\
=\frac{(1-t^2x)-(1-t^2q^2x)}{(q^2-1)(1-t^2x)(1-t^2q^2x)}=\frac{t^2x}{(1-t^2x)(1-t^2q^2x)} \;. 
\end{multline*}
Hence $a=0$.
\end{proof}

\begin{lemma}\label{lemma:F22}
\begin{align*}
F_{22} &=q^{-1}f(q^4x)^2\deb\varphi^\dag\wprod\de\varphi + q^{-1}t^2f(q^2x)^2f(q^4x)^2 \phi^*\wprod 
\phi  \\[3pt]
       &-t^2q^6\big\{f(q^4x)+f(q^2x)\big\}\dot{f}(q^3x)\deb x\wprod\de x \;.
\end{align*}
\end{lemma}

\begin{proof}
>From Lemma~\ref{lemma:B1}, Lemma~\ref{cor:B4} and equation \eqref{eq:yyy}:
\begin{align*}
q^2F_{22} &=q f(q^4x)^2\deb\varphi^\dag\wprod\de\varphi + qt^2f(q^4x)^2 
\phi^*\wprod f(x)^2 \phi \\
&+f(q^4x)^{-1}\deb f(q^4x)\wprod f(q^4x)\de f(q^4x)^{-1}+f(q^4x)\de f(q^4x)^{-1}\wprod f(q^4x)^{-1}\deb f(q^4x) \\
&+f(q^4x)\deb f(q^4x)^{-2}\wprod\de f(q^4x)+\deb f(q^4x)\wprod \big(\de f(q^4x)^{-2}\big)f(q^4x) \\
&+\de f(q^4x)\wprod f(q^4x)^{-2}\deb f(q^4x)+\deb f(q^4x)\wprod f(q^4x)^{-2}\de f(q^4x) \;. 
\end{align*}
Calling $a=f(q^4x)$, the last three lines are
\begin{gather*}
a^{-1}\deb a\wprod a\de a^{-1}+a\de a^{-1}\wprod a^{-1}\deb a \\
+a\deb a^{-2}\wprod\de a+\deb a\wprod \big(\de a^{-2}\big)a \\
+\de a\wprod a^{-2}\deb a+\deb a\wprod a^{-2}\de a 
\end{gather*}
and using the Leibniz rule they become
\begin{gather*}
a\de a^{-1}\wprod a^{-1}\deb a+\de a\wprod a^{-2}\deb a=0 \\
\deb a\wprod a^{-2}\de a+\deb a\wprod \big(\de a^{-2}\big)a=\deb a\wprod \de a^{-1} \;. 
\end{gather*}
Hence the last three lines reduce to
$$
\deb a\wprod \de a^{-1}+a^{-1}\deb a\wprod a\de a^{-1}+a\deb a^{-2}\wprod\de a
=a\deb a^{-2}\wprod\big\{\de a+a(\de a)a^{-1}\big\}
$$
Since $\deb a^{-2}=-t^2q^4\deb x$,
\begin{multline*}
a\deb a^{-2}\wprod\big\{\de a+a(\de a)a^{-1}\big\}=
-t^2q^8f(q^4x)\dot{f}(q^3x)\deb x\wprod\big\{\de x+f(q^4x)(\de x)f(q^4x)^{-1}\big\}
\\
=-t^2q^8f(q^4x)\dot{f}(q^3x)\deb x\wprod\de x\big\{1+f(q^2x)f(q^4x)^{-1}\big\}
\\
=-t^2q^8\big\{f(q^4x)+f(q^2x)\big\}\dot{f}(q^3x)\deb x\wprod\de x \;,
\end{multline*}
where we used $\de a=\dot{f}(q^5x)q^4\de x$, $(\deb x)\dot{f}(q^5x)=\dot{f}(q^3x)(\deb x)$
and similar commutation relations with $\de x$ (cf.~Corollary \ref{cor:qder} and equation
\eqref{eq:lemma36}). This concludes the proof.
\end{proof}

\begin{lemma}\label{lemma:B16}
Let $\eta'$ be the following self-dual $(1,1)$-form:
$$
\eta':=c_1
\begin{bmatrix}
q^{-\frac{1}{2}}(u^1_1)^*u^2_1 \\[3pt]
[2]^{-\frac{1}{2}}\big(q^{-1}(u^1_1)^*u^1_1-q(u^2_1)^*u^2_1\big) \\[3pt]
-q^{\frac{1}{2}}(u^2_1)^*u^1_1
\end{bmatrix}
$$
Then
$$
(\deb\varphi^\dag\wprod\de\varphi)^+ =t^2q\eta' \;,\quad\quad
(\deb x\wprod\de x)^+ =-q^{-4}x\eta' \;,\quad\quad
(\phi^*\wprod\phi)^+=q\eta' \;.
$$
\end{lemma}

\begin{proof}
This follows by comparing \eqref{eq:consequenceB}, \eqref{eq:dexdebxB} and \eqref{eq:debdep}.
Again, the commutation relations between $u^i_1$ and $(u^j_1)^*$ are obtained from the observation that
$z^i:=(u^i_1)^*$ satisfy the same commutation rules of $\Sq$ (cf.~Remark 2.2).
\end{proof}

\begin{cor}\label{cor:22+}
The self-dual part of $F_{22}$ vanishes:
$$
F_{22}^+=0
$$
\end{cor}
\begin{proof}
>From Lemma~\ref{lemma:F22} and Lemma~\ref{lemma:B16}, $F_{22}^+=t^2b\eta'$, with $b\in\Aq$ the element:
$$
b=f(q^4x)^2\big\{1-f(q^2x)^2\big\}+q^2x\big\{f(q^4x)+f(q^2x)\big\}\dot{f}(q^3x) \;.
$$
But
$$
\dot{f}(q^3x)=\frac{f(q^4x)-f(q^2x)}{(q^4-q^2)x} \;; 
$$
hence
$$
\big\{f(q^4x)+f(q^2x)\big\}\dot{f}(q^3x)=t^2f(q^2x)^2f(q^4x)^2 \:.
$$
Also $1-f(q^2x)^2=-t^2q^2xf(q^2x)^2$. Thus
$$
b=-t^2q^2xf(q^2x)^2f(q^4x)^2+t^2q^2xf(q^2x)^2f(q^4x)^2=0 \;. \vspace{-15pt}
$$
\end{proof}



\medskip

\begin{thebibliography}{10}

\bibitem{Buc86}
N.P. Buchdahl, \textit{Instantons on $\textup{CP}_2$}, J. Diff. Geom. 24 (1986) 19--52. 

\bibitem{DDL08b}
F.~D'Andrea, L.~D{\k a}browski and G.~Landi, \textit{The Noncommutative Geometry of the Quantum Projective Plane},
  Rev. Math. Phys. \textbf{20} (2008)  979--1006.

\bibitem{DL09b}
F.~D'Andrea and G.~Landi, \textit{Anti-selfdual Connections on the Quantum Projective Plane: Monopoles}, Commun. Math. Phys. 297 (2010) 841--893.

\bibitem{DL08}
F.~D'Andrea and G.~Landi, \textit{Geometry of the quantum projective
  plane}, Noncommutative Structures in Mathematics and Physics, 5th ECM
  Satellite Conf. Proceedings, Royal Flemish Acad. (Brussels), 2008, pp.~85--102.

\bibitem{Don84}
S.K. Donaldson, \textit{Vector bundles on the flag manifolds and the Ward correspondence}, in Geometry Today, Progress in Math. 60, Birkh{\"a}user, 1985, pp. 109--119.

\bibitem{DK90}
S.K. Donaldson and P.B. Kronheimer, \textit{The geometry of four-manifolds},
Oxford Univ.~Press, 1990.

\bibitem{Gro90}
D. Groisser, \textit{The geometry of the moduli space of $\,\textup{CP}^2\!$ instantons},
Invent. Math. 99 (1990) 393--409.

\bibitem{Hab92}
L. Habermann, \textit{A family of metrics on the moduli space of $\,\textup{CP}^2\!$ instantons},
Commun. Math. Phys. 149 (1992) 209--216. 

\bibitem{KS97}
A.~Klimyk and K.~Schm{\"u}dgen, \textit{Quantum groups and their
  representations}, Springer, 1997.

\bibitem{VS91}
L.~Vaksman and Ya. Soibelman, \textit{The algebra of functions on the quantum
group $SU(n+1)$ and odd-dimensional quantum spheres}, Leningrad Math.~J. 2 (1991) 1023--1042.

\bibitem{Wel80}
R.O.~Wells, \textit{Differential analysis on complex manifolds}, GTM \textbf{65}, Springer, 1980.

\bibitem{Wor96}
S.L.~Woronowicz,
\textit{From multiplicative unitaries to quantum groups}, 
Int. J. Math. 7 (1996) 127--149.

\end{thebibliography}
\end{document}